\documentclass[11pt, reqno]{amsart}
\usepackage{amsmath,amssymb,amsfonts,mathrsfs,verbatim,enumitem,pstricks,amsthm, graphicx,setspace}
\usepackage[margin=2.5cm]{geometry}
\usepackage[all]{xy}
\usepackage{centernot, xr, accents}
\usepackage{hyperref,comment}
\usepackage[T1]{fontenc}
\usepackage[utf8]{inputenc}
\usepackage{amsmath,amssymb,mathtools}
\usepackage[all]{xy}
\usepackage{hyperref}
\usepackage{verbatim}
\allowdisplaybreaks
\usepackage[bottom]{footmisc}
\usepackage{lipsum,cleveref}

\usepackage{hyperref}
\hypersetup{
colorlinks,
citecolor=teal,
filecolor=black,
linkcolor=blue,
urlcolor=black
}

\theoremstyle{plain}

\newtheorem{thm}{Theorem}[section]
\newtheorem{cor}[thm]{Corollary}
\newtheorem{lmm}[thm]{Lemma}
\newtheorem{prp}[thm]{Proposition}

\theoremstyle{definition}

\newtheorem{rem}[thm]{Remark}

\newtheorem{notn}[thm]{Notation}

\theoremstyle{definition}

\numberwithin{equation}{section}

\begin{document}
\title[On the isotropy group of monomial derivations]{On the isotropy group of monomial derivations}
\author[S. C. Mishra]{Sumit Chandra Mishra}
\address{Indian Institute of Technology Indore, Simrol, Khandwa Road, Indore 453 552, India.}
\email{sumitcmishra@gmail.com, sumitcmishra@iiti.ac.in}
\author[D. Mondal]{Dibyendu Mondal}
\address{Indian Institute of Technology Indore, Simrol, Khandwa Road, Indore 453 552, India.}
\email{mdibyendu07@gmail.com; mdibyendu@iiti.ac.in}
\author[P. Shukla]{Pankaj Shukla}
\address{Indian Institute of Technology Indore, Simrol, Khandwa Road, Indore 453 552 India}
\email{Pankajshuklashvm23@gmail.com, phd2201141002@iiti.ac.in}
\let\thefootnote\relax

\begin{abstract}
In this article, we characterize the isotropy groups of certain special monomial and Jouanolou-type derivations of polynomial rings over fields of characteristic zero. Under suitable conditions, we determine the structure of these isotropy groups.
\end{abstract}
\date{\today}
\maketitle

\section{Introduction}\label{intro}
Throughout this article, all rings are assumed to be of characteristic zero and  commutative with identity. For any ring $R$, we denote its set of units by $R^*$. Let $k$ be a field and $n\in \mathbb{N}$. By $\mu_n$, we denote the group of all $n$-th roots of unity in $k$. We write $k[x_1,\ldots ,x_n]$ for the polynomial ring in $n$ variables over $k$. When $n$ is implicit, we write $k[X]$ instead of $k[x_1,\ldots,x_n]$. We denote by Aut$(k[X])$ the group of all $k$-automorphisms of $k[X]$. A $k$-derivation on $k[X]$ is a $k$-linear map $d : k[X] \to k[X]$ such that $d(fg) =gd(f) +fd(g)$ for all $f,g\in k[X]$. 

For a given derivation $d$ of $k[X]$, the \emph{isotropy group} of $d$, denoted by Aut$(k[X])_d$, is the subgroup of Aut$(k[X])$ consisting of all $k$-automorphisms commuting with $d$. The isotropy group of $d$ is the stabilizer of $d$ under the conjugation action of Aut$(k[X])$ on the set of all $k$-derivations of $k[X]$. The isotropy group plays a fundamental role in understanding the symmetries of the derivations, see \cite{BP21}. Geometrically, since every $k$-derivation $d$ of $k[X]$ determines a polynomial vector field on the affine space $\mathbb{A}_k^n$, the elements of Aut$(k[X])_d$ correspond to the polynomial symmetries of the associated vector field. 

In recent years, the study of isotropy groups of derivations has received considerable attention, with particular emphasis on the isotropy groups of locally nilpotent and simple derivations of $k[X]$. In case of simple derivation, Mendes and Pan showed that the isotropy group of any simple derivation of $k[x_1,x_2]$ is trivial, see \cite{MP17}. Additionally, Bertoncello and Levcovitz  proved that the isotropy group of a simple Shamshuddin derivation of any $k[X]$ is trivial, see \cite{BN}; converse of their result was subsequently proved by Yan in \cite{Yan24}. Further study of isotropy groups of simple derivations can be found in \cite{Bal26, MMS26, MPR24, Yan22}. Danielewski-type varieties form an important class of varieties in the study of cancellation problems in affine geometry. The isotropy group of a locally nilpotent derivation on Danielewski surfaces was first studied by Baltazar and Veloso in \cite{Bal21}. Subsequently, isotropy groups of locally nilpotent derivations on Danielewski-type varieties were investigated in \cite{Bal25, DL23}. Further study of isotropy group of non-simple derivations and locally finite derivations of $k[x_1,x_2]$ were studied in \cite{RK23} and \cite{CV26}, respectively. In this article, we study isotropy group of certain classes of monomial derivations.

Let $d$ be a $k$-derivation of $k[X]=k[x_1,\ldots ,x_n]$.  The derivation $d$ is called a \emph{monomial derivation} if, for each $i\in \{1,\ldots ,n\}$, $$d(x_i)=\alpha_i x_1^{\beta_{i1}}......x_n^{\beta_{in}}$$ where $\alpha_i\in k^{\ast}$ and each $\beta_{ij}$ is a non-negative integer. A non-constant polynomial $f\in k[X]$ is said to be a {\it Darboux polynomial} of $d$ if $d(f)=\Lambda f$ for some $\Lambda\in k[X]$. If a non-constant polynomial $f\in k[X]^d$ (the ring of constants of $d$) then $f$ is a Darboux polynomial of $d$. Consequently, if $d$ has no Darboux polynomial then the ring of constant $k[X]^d$ is trivial, i.e., $k[X]^d=k$. But the converse does not hold in general, see \cite[Lemma 10.4.1]{n94}. For a monomial derivation $d$, the determination of the field of rational constants $k(X)^d$ is a fundamental problem. Indeed, nonconstant elements of $k(X)^d$ are precisely the rational first integrals of the differential system associated with $d$, see \cite{NZ}. A monomial derivation $d$ is called \emph{special} if either $d$ has no Darboux polynomial, or any irreducible Darboux polynomial of $d$ is, up to multiplication by a non-zero scalar, one of the coordinate variables. In \cite{OJN}, the authors proved that, for certain monomial derivations, the triviality of the field of rational constants can be characterized entirely in terms of Darboux polynomials. 
 
In Section \ref{sec 3}, we study the isotropy group of special monomial derivations with $\beta_{ii}\geq 1$ for all $i=1,\ldots,n$. Let $A=[\beta_{ij}]-I$ be the matrix associated with the derivation $d$. We give a characterization of the isotropy group of such special monomial derivations with det$(A)\neq 0$ (see \Cref{iso-thm}). In particular, if $\alpha_i=1$ for all $i=1,\ldots ,n$, then we prove that Aut$(k[X])_d\cong (\mu_{d_1}\times\cdots\times\mu_{d_n})\rtimes K$, where $d_i\in \mathbb{N}$ are such that $\text{diag}\{d_1,\ldots,d_n\}$ is the Smith normal form of $A$, and $K$ is the centralizer of $A$ in the symmetric group $S_n$, viewed as the group of all $n\times n$ permutation matrices over $k$ (see \Cref{sem-dir-product}). 

In Section \ref{sec 4}, we study the Jouanolou-type derivations $J:=J(n,r,\sigma)=\sum_{i=1}^n x_i^{\,r}\,\partial_{x_{\sigma(i)}}$ of the polynomial ring $k[x_1,\ldots,x_n]$, where $\sigma \in S_n$ is a non-identity permutation. For $r\geq 2$ and $n\geq 2$, we obtain a complete characterization of the isotropy group of $J$ (see \Cref{J main}). 
In \Cref{thm:isotropy}, we prove that $$\text{Aut}(k[X])_J\cong\prod_m\left(\mu_{r^m-1}^{l_m}\rtimes((\mathbb{Z}/m\mathbb{Z})^{l_m}\rtimes S_{l_m})\right),$$ where $m$ ranges over the cycle lengths occurring in the disjoint cycle decomposition of $\sigma$, and $l_m$ denotes the number of $m$-cycles appearing in the decomposition. We also describe the corresponding semidirect product structure.

\section{Preliminaries}\label{prelim}
  Let $R$ be a commutative ring with identity. For a positive integer $n$, we write $M_n(R)$ for the ring of all $n\times n$ matrices over $R$. We denote by $I$ the identity matrix in $M_n(R)$. Let $a_1,\ldots,a_n\in R$. We denote by diag$\{a_1,\ldots,a_n\}$ the diagonal matrix in $M_n(R)$ whose diagonal entries are $a_1,\ldots,a_n$.
  The symbol $\delta_{ij}$ denotes the Kronecker delta function for $i,j\in \{1,\ldots ,n\}$.
  Let $k[X]:=k[x_1,\ldots,x_n]$ be the polynomial ring in $n$ variables over $k$, and let $k(X):=k(x_1,\ldots,x_n)$ denote its field of fractions. Let $d$ be a $k$-derivation of $k[X]$. Then $d$ extends uniquely to a $k$-derivation of $k(X)$, which we also denote by $d$. We write $k(X)^d:=\{\varphi\in k(X)|\, d(\varphi)=0\}$ for the field of rational constants of $d$. Moreover, we say that the field of rational constants of $d$ is trivial if and only if $k(X)^d=k$.

  Let $(G,\cdot)$ be a group. Then the opposite group of $G$, denoted by $(G^{op},\star)$, is a group with underlying set same as that of $G$ with  the binary operation given by $g \star h=h\cdot g$ for all $g,h \in G^{op}$. Note that $G$ is isomorphic to $G^{op}$ since the map $g \mapsto g^{-1}$ for all $g\in G$ gives an isomorphism.

\subsection{Monomial derivations} Recall that a derivation $d$ of $k[X]$ is said to be a \emph{monomial derivation} if $d(x_i)=\alpha_i{x_1}^{\beta_{i1}}\dots{x_n}^{\beta_{in}}$ for $i=1,\dots,n$, where each $\beta_{ij}$ is a non-negative integer and $\alpha_i\in k^{\ast}$ for all $i\in \{1,\ldots,n\}$.
   
   A monomial derivation $d$ of $k[X]$ is called {\it special} if either $d$ has no Darboux polynomials, or every irreducible Darboux polynomial of $d$ is, up to multiplication by a non-zero scalar, one of the coordinate variables $\{x_1,\ldots,x_n\}$.
   
   Note that if $d$ is without Darboux polynomials, then $k(X)^d=k$. The following theorem provides a characterization of monomial derivations with trivial field of rational constants in terms of special derivations.
   
   \begin{thm}\label{lem1}\cite[Theorem 3.3]{OJN}
    Let $d$ be a derivation of a polynomial ring $k[X]=k[x_1,\dots,x_n]$, where $k$ is a field of characteristic zero. Assume that $$d(x_i)=\alpha_i{x_1}^{\beta_{i1}}\dots{x_n}^{\beta_{in}}$$ for $i=1,\dots,n$, where each $\alpha_i\in k^*$ and each $\beta_{ij}$ is non-negative integer. Denote by $A$ the $n\times n$ matrix $[\beta_{ij}]-I,$ and let $w_d=det(A)$. If $w_d\neq 0$, then the following conditions are equivalent.
    \begin{itemize}
        \item[(1)] Either $d$ is without Darboux polynomial or all irreducible Darboux polynomials of $d$ belong to the set $\{x_1,\dots,x_n\}.$
        \item[(2)] The field $k(X)^d$ is trivial.
    \end{itemize}
 \end{thm}
 
 \subsection{Smith normal form} \cite{Ku,MR,stan}
 Let $R$ be a commutative ring with identity, and let $M\in M_n(R)$. Recall that the elementary row and column operations on $M$ over $R$ consist of interchanging two rows (resp. columns) with one another, adding a multiple of one row (resp. column) to another row (resp. column), and multiplying a row (resp. column) by a unit of $R$. Furthermore, elementary row (resp. column) operation on $M$  corresponds to left (resp. right) multiplication of $M$ by invertible matrices over $R$. Consequently, two matrices $M,N\in M_n(R)$ are said to be {\it equivalent} if there exist invertible matrices $P$ and $Q$ in $M_n(R)$ such that $PMQ=N$.
 
 Let $D\in M_n(R)$ denote a diagonal matrix. The matrix $D$ is called a {\it diagonal form} of $M$ if the matrices $M$ and $D$ are equivalent. A diagonal form $D=\text{diag}\{d_1,\ldots,d_r,0,\ldots, 0\}$ of $M$ is called a \emph{Smith normal form} (SNF) of $M$ if $d_i$ divides $d_{i+1}$ for all $1\leq i \leq r-1$. 
  
 Suppose that $M$ is a unimodular matrix over $R$, i.e., det($M$) is a unit in $R$, then the identity matrix $I$ is a diagonal form of $M$ as $M^{-1}MI=I$. Thus, diag$\{1,1,\dots, 1\}$ is a Smith normal form of $M$. 
 
 Next, we recall some results regarding Smith normal forms over principal ideal domains. The following theorem establishes the existence of a Smith normal form over a principal ideal domain. 
 \begin{thm}\cite[Theorem 3.8]{Jac}\label{snf}
     Let $R$ be a principal ideal domain and $M\in M_n(R)$, then the diagonal matrix diag$\{d_1,d_2\dots,d_r,0,\ldots,0\}$ is a Smith normal form of $M$, where $d_i\neq0$ for all $1\leq i\leq r$ and $d_i$ divides $d_j$ if $i\leq j.$
 \end{thm}
Moreover, the Smith normal form of $M$, in \Cref{snf}, is unique up to multiplication of each non-zero diagonal entry $d_i$ by a unit of $R$. 
 
 \begin{thm}\cite[Proposition 8.1]{MR}\label{compute d_i}
 Let $R$ be a unique factorization domain and $M\in M_n(R)$. Assume that diag$\{d_1,\dots,d_r,0\ldots,0\}$ is a Smith normal form of $M$, where $1\leq r\leq n$. Then, for all $1\leq l\leq n$, we get that $d_1d_2\cdots d_l$ is equal to the gcd of all $l\times l$ minors of $M$, where $d_l=0$ for $l>r$, and with the convention that if all $l\times l$ minors are zero then their gcd is zero.
 \end{thm}

 The proofs of \Cref{snf} and \Cref{compute d_i} can also be found in \cite[Theorem 17]{Ku} and \cite[Theorem 2.4]{stan}, respectively.

 Next, we introduce some notations which will be followed throughout this article. For a matrix $M=[m_{ij}]\in$ M$_n(\mathbb{Z})$ and $c=(c_1,\ldots ,c_n)\in k^n$, we define $$c^M:=\left(\prod_{i=1}^nc_i^{m_{1i}},\dots,\prod_{i=1}^nc_i^{m_{ni}}\right) \in k^n.$$
 We denote by $1$ the identity element $(1,\dots, 1)$ of the group $(k^\ast)^n$. Let $c=(c_1,\ldots,c_n)\in (k^\ast)^n$, we denote the inverse of $c$ in $(k^\ast)^n$ by $c^{-1}$, where $c^{-1}:=(c_1^{-1},\ldots,c_n^{-1})$. 
\begin{lmm}\label{lem c^M}
Let $M,N\in M_n(\mathbb Z)$ and $c,d\in (k^\ast)^n$. Then
\begin{enumerate}
\item $(c d)^M=c^M d^M$;
\item $(c^{-1})^M=(c^M)^{-1}$;
\item $(c^M)^N=c^{NM}$;
\item $c^{I}=c$, where $I\in M_n(\mathbb Z)$ is the identity matrix.
\end{enumerate}
\end{lmm}
\begin{proof}
Assertions (1) and (2) follow directly from the definition of $c^M$. Let $M=[m_{ij}]$, $N=[n_{ij}]$, and $c=(c_1,\ldots,c_n)\in (k^{\ast})^n$. To prove (3), observe that the $s$-th component of $(c^M)^N$ is $$\prod_{j=1}^{n}\left(\prod_{i=1}^{n}c_i^{m_{ji}}\right)^{n_{sj}}=\prod_{i=1}^{n}c_i^{\sum_{j=1}^{n}n_{sj}m_{ji}}.$$ Since $\sum_{j=1}^{n} n_{sj}m_{ji}=(NM)_{si}$, it follows that the $s$-th component of $(c^M)^N$ coincides with the $s$-th component of $c^{NM}$. Hence $(c^M)^N=c^{NM}$.

Finally, since $I=[\delta_{ij}]$, we have $(c^{I})_s=\prod_{i=1}^{n}c_i^{\delta_{si}}=c_s$ for every $s=1,\ldots,n$. Therefore $c^{I}=c$.
\end{proof}

Let $A\in M_n(\mathbb{Z})$ with det$(A)\neq 0$. By \Cref{snf}, $A$ admits a Smith normal form $D=$ diag$\{d_1,\ldots,d_n\}$. Since det$(A)\neq 0$, we have $d_i\neq 0$ for all $i=1,\ldots ,n$. Moreover, the diagonal entries $d_i$ of $D$ are unique up to multiplication by units of $\mathbb{Z}$. Therefore, by choosing each $d_i$ to be positive, we obtain a unique Smith normal form of $A$.

\begin{lmm}\label{struc lem H_A}
Let $A\in M_n(\mathbb Z)$ with det$(A)\neq 0$, and let $D=\text{diag}\{d_1,\dots,d_n\}$ be the Smith normal form of $A$ with $d_i\in \mathbb{N}$ for all $i=1,\ldots ,n$. Let $H_A=\{c\in (k^*)^n|\,c^A=1\}$ be the subgroup of $(k^*)^n$. Then $H_A\cong \mu_{d_1}\times\cdots\times\mu_{d_n},$ where $\mu_{d_i}$ denotes the group of $d_i$-th roots of unity in $k$.
\end{lmm}

\begin{proof} 
Since $D=\text{diag}\{d_1,\dots,d_n\}$ is the Smith normal form of $A$, there exist invertible matrices $P,Q\in M_n(\mathbb Z)$ such that $PAQ=D$. Let $H_A=\{c\in (k^\ast)^n|\,c^A=1\}$ and $H_D=\{u\in (k^\ast)^n|\,u^D=1\}$ be the subgroups of $(k^*)^n$ associated with $A$ and $D$, respectively. Let $u\in H_D$. Then $u^D=1$. Since $D=PAQ$, it follows that $u^{PAQ}=1$. Applying $(-)^{P^{-1}}$ to both sides and using \Cref{lem c^M}(3), we have $(u^{PAQ})^{P^{-1}}=u^{P^{-1}PAQ}=u^{AQ}=1$. Now let $c=u^Q$. By \Cref{lem c^M}(3), $c^A=(u^Q)^A=u^{AQ}=1$. Hence $c\in H_A$. Thus the map $\varphi:H_D\longrightarrow H_A$, given by $\varphi(u)=u^Q$, is well defined. 

Next, we show that $\varphi$ is a group homomorphism. Let $u,v\in H_D$. Then $\varphi(uv)=(uv)^Q$. If $Q=[q_{ij}]$, then the $k$-th component of $(uv)^Q$ is $$\prod_{i=1}^n(u_iv_i)^{q_{ki}}=\left(\prod_{i=1}^nu_i^{q_{ki}}\right)\left(\prod_{i=1}^nv_i^{q_{ki}}\right).$$ Hence $(uv)^Q=u^Qv^Q$, and therefore $\varphi(uv)=\varphi(u)\varphi(v)$. Thus, $\varphi$ is a group homomorphism.

Next, we show that $\varphi$ is injective. Suppose that $\varphi(u)=\varphi(v)$. Then $u^Q=v^Q$. Applying $(-)^{Q^{-1}}$ to both sides and using \Cref{lem c^M}(3) and (4), we have $u=(u^Q)^{Q^{-1}}=(v^Q)^{Q^{-1}}=v$. Thus $\varphi$ is injective. To prove surjectivity, let $c\in H_A$, and define $u=c^{Q^{-1}}$. Then, by \Cref{lem c^M}(3) and (4), $u^Q=c$. Since $c\in H_A$, we have $c^A=1$. Hence $(u^Q)^A=1$, and therefore $u^{AQ}=1$.
Applying $(-)^{P}$ to both sides and using \Cref{lem c^M}(3), we obtain that $(u^{AQ})^P=u^{PAQ}=1$. Using $PAQ=D$, it follows that
$u^D=1$. Thus $u\in H_D$, and $\varphi(u)=u^Q=(c^{Q^{-1}})^Q=c^{QQ^{-1}}=c$. Therefore, $\varphi$ is surjective. Hence, $H_A\cong H_D$.

Since, $D=\text{diag}\{d_1,\dots,d_n\}$, we have $u^D=(u_1^{d_1},\dots,u_n^{d_n})$. Therefore, $u\in H_D$ if and only if $u_i^{d_i}=1$ for each $i=1,\dots,n$. Hence $H_D=\mu_{d_1}\times\cdots\times\mu_{d_n}$. Consequently, $H_A\cong \mu_{d_1}\times\cdots\times\mu_{d_n}$.
\end{proof}

\section{Special monomial derivations with $\beta_{ii}\geq 1$}\label{sec 3}

Let $n\in \mathbb{N}$. In this section, we study a class of special monomial derivations $d$ of $k[X]:=k[x_1,\dots, x_n]$ satisfying the property that $x_i$ divides $d(x_i)$ for all $i$, $1\leq i \leq n$. We give necessary and sufficient conditions for a $k$-automorphism of $k[X]$ to be in the isotropy group Aut$(k[X])_d$. Furthermore, in \Cref{final-iso-thm}, under the assumption that $\alpha_i=1$ for all $i$, we describe the structure of the group Aut$(k[X])_d$.

Note that for every such derivations $d$, $x_i$ is a Darboux polynomial of $d$ for all $1\leq i\leq n$. If we consider $\alpha_{i}=1$, $\beta_{ii}=1$ for all $ 1\leq i\leq n$, $\beta_{i(i-1)}=1$ for all $2\leq i\leq n$, $\beta_{1n}=1$, and $\beta_{ij}=0$ otherwise, then we get the Lotka-Volterra derivation.
As a consequence of Theorem \ref{final-iso-thm} (see Remark \ref{LV-remark}), we recover a result of Kour and Rewri (\cite{RK}) which states that the isotropy group of the Lotka-Volterra derivation is isomorphic to $\mathbb{Z}/n\mathbb{Z}$. 

Throughout this section, for any permutation $\sigma\in S_n$,  we denote by $P_{\sigma}:=[(P_{\sigma})_{ij}]\in M_n(k)$ the permutation matrix associated with $\sigma$, where $(P_{\sigma})_{ij}=\delta_{i\sigma(j)}$ for all $1\leq i,j \leq n.$ Note that for $\sigma,\tau \in S_n$, $(P_{\tau}P_{\sigma})_{ij}=\sum_{r=1}^{n}(P_{\tau})_{ir}(P_{\sigma})_{rj}=\sum_{r=1}^{n}\delta_{i\tau(r)}\delta_{r\sigma(j)}=\delta_{i\tau(\sigma(j))}=(P_{\tau\sigma})_{ij}$, for all $1\leq i,j \leq n$, that is, $P_{\tau\sigma}=P_{\tau}P_{\sigma}$. Consequently, $(P_{\sigma})^{-1}=P_{\sigma^{-1}}.$

 \begin{thm}\label{iso-thm}
Let $k$ be a field of characteristic zero, and let $d$ be the $k$-derivation of $k[x_1,\dots,x_n]$ defined by
$$d(x_i)= \alpha_i x_1^{\beta_{i1}}\cdots x_n^{\beta_{in}}, \quad i=1,\dots,n,$$ where $\alpha_i \in k^{\ast}$ and $\beta_{ij} \in \mathbb{Z}_{\ge 0}$ for all $1\leq i,j \leq n$. Let $A=[a_{ij}]$ be given by $a_{ij} = \beta_{ij} - \delta_{ij}$, and let $w_d := \det(A)$. Furthermore, assume that $\beta_{ii} \geq 1$ for all $1\leq i\leq n$, $w_d \neq 0$, and $k(X)^d=k$. Let $\rho \in \text{Aut}(k[X])$. Then $\rho \in \text{Aut}(k[X])_d$ if and only if there exist some $\sigma \in S_n$ and $(c_1,\dots,c_n)\in (k^*)^n$ such that $$\rho(x_i)=c_i x_{\sigma(i)}, \,\,\,P_\sigma A=AP_\sigma,\,\,\,\text{and}\,\,\,\prod_{t=1}^n c_t^{a_{it}} =\frac{\alpha_{\sigma(i)}}{\alpha_i} \,\,\, \text{for all}\; 1\leq i\leq n.$$
\end{thm}
\begin{proof}
     Let $\rho\in$ Aut$(k[X])_d$. Since $\beta_{ii}\geq 1$ for all $1\leq i \leq n$, we have $d(x_i)=\alpha_i {x_1}^{\beta_{i1}}\dots{x_n}^{\beta_{in}}=\lambda_ix_i$ where the cofactor $\lambda_i=\alpha_i{x_1}^{\beta_{i1}}\dots{x_i}^{\beta_{ii}-1}\dots{x_n}^{\beta_{in}}\in k[X].$ Thus $x_i$ is a Darboux polynomial of $d$ for all $i=1,\dots,n$. Since $d\rho=\rho d$, we have that for all $i=1,\dots,n$,
     \begin{gather}\label{eq1}
     d\rho(x_i)=\rho d(x_i)=\rho(\lambda_ix_i),\,\text{that is},\,\,d(\rho(x_i))=\rho(\lambda_i)\rho(x_i).
     \end{gather}
      Thus $\rho(x_i)$ is a Darboux polynomial for all $1\leq i\leq n$.
      
      Note that $\rho(x_i)$ is irreducible since $\rho$ is an automorphism of $k[X]$. Also, we have that $w_d\neq 0$ and $k(X)^d$ is trivial. Hence, by \Cref{lem1}, it follows that $\rho$ permutes the variables $x_1,\dots,x_n$ up to nonzero scalar multiples, that is, there exist $\sigma \in S_n$ and $c=(c_1,\dots,c_n)\in (k^*)^n$ such that
      \begin{equation}\label{eq2}
          \rho(x_i)=c_ix_{\sigma(i)}
      \end{equation}
    for all $i=1,\ldots,n$. We denote this automorphism by $\rho_{c,\sigma}$. Furthermore, we have 
    \begin{equation}\label{eq3}
    d(\rho_{c,\sigma}(x_i))=d(c_ix_{\sigma(i)})=c_id(x_{\sigma(i)})=c_i\lambda_{\sigma(i)}x_{\sigma(i)}
    \end{equation}
    and 
    \begin{equation}\label{eq4}
    \rho_{c,\sigma}(d(x_i))=\rho_{c,\sigma}(\lambda_i x_i)=\rho_{c,\sigma}(\lambda_i)\rho_{c,\sigma}(x_i)=\rho_{c,\sigma}(\lambda_i)\,c_ix_{\sigma(i)}.
     \end{equation}
     Thus, from the equality $d\rho_{c,\sigma}=\rho_{c,\sigma}d$, and Equations \ref{eq3} and \ref{eq4}, it follows that
     \begin{equation}\label{eq5}
         \lambda_{\sigma(i)}=\rho_{c,\sigma}(\lambda_i).
     \end{equation}
     Recall that $\lambda_i=\alpha_i {x_1}^{\beta_{i1}}\dots{x_i}^{\beta_{ii}-1}\dots{x_n}^{\beta_{in}} = \alpha_i \prod_{t=1}^n x_t^{\beta_{it}-\delta_{it}},$ where $\delta_{it}$ is the Kronecker delta function. Since $a_{it}=\beta_{it}-\delta_{it}$ for all $i,t=1,\ldots,n$, then $\lambda_i=\alpha_i \prod_{t=1}^n x_t^{a_{it}}$ for all $i,t=1,\ldots,n$. Therefore
    $$\rho_{c,\sigma}(\lambda_i) =\alpha_i \prod_{t=1}^n (\rho_{c,\sigma}(x_t))^{a_{it}}=\alpha_i\prod_{t=1}^n(c_tx_{\sigma(t)})^{a_{it}}=\alpha_i\left(\prod_{t=1}^nc_t^{a_{it}}\right)\left(\prod_{t=1}^n x_{\sigma(t)}^{a_{it}}\right),$$ for all $i$.
    Re-indexing, via $s=\sigma(t)$, we get that $$\rho_{c,\sigma}(\lambda_i)=\alpha_i\left(\prod_{t=1}^nc_t^{a_{it}}\right)\left(\prod_{s=1}^nx_s^{a_{i\sigma^{-1}(s)}}\right).$$
    On the other hand, we have $\lambda_{\sigma(i)}=\alpha_{\sigma(i)} \prod_{s=1}^n x_s^{a_{\sigma(i)s}}$ for all $i$. Therefore, by Equation~\ref{eq5},
    \begin{equation}\label{eq6}
    \alpha_{\sigma(i)}\prod_{s=1}^nx_s^{a_{\sigma(i)s}}=\alpha_i\left(\prod_{t=1}^nc_t^{a_{it}}\right)\left(\prod_{s=1}^n x_s^{a_{i\sigma^{-1}(s)}}\right),
    \end{equation} for all $i.$
    Now, comparing the exponents of all $x_s$ and coefficients of both sides of Equation~\ref{eq6}, we get that \begin{equation}\label{eq7}
         a_{\sigma(i)s}=a_{i\sigma^{-1}(s)}
    \end{equation} for all $i,s=1,\ldots ,n$, and
    \begin{equation*}
         \prod_{t=1}^n c_t^{a_{it}} = \frac{\alpha_{\sigma(i)}}{\alpha_i},
    \end{equation*} for all $i,t=1,\ldots ,n$, respectively.
    
    Next, we show that the matrix $P_{\sigma}$ and the matrix $A$ commute. Reindexing Equation~\ref{eq7}, via replacing $i$ with $\sigma^{-1}(i)$ and $s$ with $\sigma(s)$, we have that  $a_{\sigma(\sigma^{-1}(i))\sigma(s)}=a_{\sigma^{-1}(i)\sigma^{-1}(\sigma(s))}$, that is,
    \begin{equation}\label{eq 7'}
      a_{i\sigma(s)}=a_{\sigma^{-1}(i)s},  
    \end{equation}
     for all $1\leq i,s\leq n$. Recall that $P_{\sigma}:=[(P_{\sigma})_{ij}]=[\delta_{i\sigma(j)}]$ denotes the permutation matrix associated with $\sigma$. Then the $is$-th entry of the matrix $P_{\sigma}A$ is given by $(P_{\sigma}A)_{is}= \sum_{r=1}^{n}(P_{\sigma})_{ir}a_{rs}=\sum_{r=1}^{n}\delta_{i\sigma(r)}a_{rs}=a_{\sigma^{-1}(i)s}$. Similarly, the $is$-th entry of the matrix $AP_{\sigma}$ is given by $(AP_{\sigma})_{is}=\sum_{r=1}^{n}a_{ir}(P_{\sigma})_{rs}=\sum_{r=1}^{n}a_{ir}\delta_{r\sigma(s)}=a_{i\sigma(s)}$. From Equation~\ref{eq 7'}, we have that $a_{i\sigma(s)}=a_{\sigma^{-1}(i)s}$ for all $i,s$, that is, $(AP_{\sigma})_{is}=(P_{\sigma}A)_{is}$ for all $1\leq i,s \leq n$. Hence $P_{\sigma}A=AP_{\sigma}$.
     
     Conversely, suppose that $\sigma \in S_n$ satisfies $P_\sigma A = A P_\sigma$, and let $(c_1,\dots,c_n)\in (k^*)^n$ satisfy $\prod_{t=1}^n c_t^{a_{it}}=\frac{\alpha_{\sigma(i)}}{\alpha_i}$ for all $i$. Consider the automorphism $\rho \in$ Aut$(k[X])$ defined by $\rho(x_i):=c_i x_{\sigma(i)}$ for all $i=1,\ldots ,n$. We claim that $\rho$ commutes with $d$. For all $i=1\ldots,n$, we have that
    \begin{equation}\label{conequ1}
\rho(d(x_i))=\rho\left(\alpha_i\prod_{t=1}^nx_t^{\beta_{it}}\right)=\alpha_i\prod_{t=1}^n(c_tx_{\sigma(t)})^{\beta_{it}}=\alpha_i \left(\prod_{t=1}^n c_t^{\beta_{it}}\right)\left(\prod_{t=1}^n x_{\sigma(t)}^{\beta_{it}}\right).
    \end{equation}
    Since $a_{it}=\beta_{it}-\delta_{it}$ for all $i,t=1,\ldots,n$, using the defining condition on $c=(c_1,\dots,c_n)$, we have that
    \begin{equation}\label{conequ2}
        \prod_{t=1}^n c_t^{\beta_{it}}=c_i \cdot \prod_{t=1}^n c_t^{a_{it}}= \frac{c_i\,\alpha_{\sigma(i)}}{\alpha_i},
    \end{equation} for all $i.$
    Moreover, reindexing the monomial factor via $s=\sigma(t)$, we have that
    \begin{equation}\label{conequ3}
    \prod_{t=1}^nx_{\sigma(t)}^{\beta_{it}}=\prod_{s=1}^nx_s^{\beta_{i\sigma^{-1}(s)}},
    \end{equation} for all $i.$
    We now use the condition on the matrix $A$. The equality $P_\sigma A = A P_\sigma$ implies that $a_{\sigma^{-1}(i)s}=a_{i\sigma(s)}$ for all $1\leq i,s\leq n$. On reindexing, we have that $a_{i\sigma^{-1}(s)}=a_{\sigma(i)s}$ for all $i$ and $s$. Substituting $a_{ij}=\beta_{ij}-\delta_{ij}$, we have $\beta_{i\sigma^{-1}(s)}-\delta_{i\sigma^{-1}(s)}=\beta_{\sigma(i)s}-\delta_{\sigma(i)s}$. Since $\delta_{i\sigma^{-1}(s)} = \delta_{\sigma(i)s}$, it follows that $\beta_{i\sigma^{-1}(s)}=\beta_{\sigma(i)s}$. Thus,
    \begin{equation}\label{conequ4}
         \prod_{s=1}^n x_s^{\beta_{i\sigma^{-1}(s)}}=\prod_{s=1}^n x_s^{\beta_{\sigma(i)s}},
    \end{equation} for all $i.$
    Using Equation~\ref{conequ2}, Equation~\ref{conequ3} and Equation~\ref{conequ4}, Equation~\ref{conequ1} becomes
    \begin{equation}\label{conequ5}
        \rho(d(x_i))=c_i\alpha_{\sigma(i)}\prod_{s=1}^n x_s^{\beta_{\sigma(i)s}},
    \end{equation} for all $i.$
    Since $d(x_i)=\alpha_i \prod_{s=1}^n x_s^{\beta_{is}}$ for all $i$, we have that
    \begin{equation}\label{conequ6}
        d(x_{\sigma(i)})=\alpha_{\sigma(i)} \prod_{s=1}^n x_s^{\beta_{\sigma(i)s}},
    \end{equation} for $\sigma\in S_n$ and for all $i$. Then it follows from Equation~\ref{conequ5} and Equation~\ref{conequ6} that
    \begin{equation}\label{conequ7}
        \rho(d(x_i))=c_i\,d(x_{\sigma(i)}).
    \end{equation}
    Since $\rho(x_i)=c_ix_{\sigma(i)}$ for all $i$, it follows from Equation~\ref{conequ7} that $\rho(d(x_i))=c_i\,d(x_{\sigma(i)})=d(c_ix_{\sigma(i)})=d(\rho(x_i))$ for all $i$. Hence $d\rho=\rho d$, and therefore $\rho \in$ Aut$(k[X])_d$.
\end{proof}
Next, we show that if $\alpha_i=1$ for all $i=1,\ldots,n$, then Aut$(k[X])_d$ itself admits a natural semidirect product structure. Let $G$ be a group, and let $G_1$ and $G_2$ be subgroups of $G$. Recall that $G$ is called the \emph{internal semidirect product} of $G_1$ and $G_2$, denoted by $G=G_1\rtimes G_2$, if $G_1$ is a normal subgroup of $G$, $G_1\cap G_2$ is the trivial subgroup of $G$, and $G=G_1G_2$. In this case, every element of $G$ can be written uniquely in the form $g_1g_2$, where $g_1\in G_1$ and $g_2\in G_2$. Furthermore, for $g_1,g'_1\in G_1$ and $g_2,g'_2\in G_2$, we have $(g_1g_2)(g'_1 g'_2)= g_1(g_2 g'_1 g_2^{-1})g_2 g'_2$.

\begin{thm}\label{sem-dir-product}
Let $d$ and $A$ be as in Theorem~\ref{iso-thm}. Furthermore, assume that $\alpha_i=1$ for all $i=1,\dots,n$. Let us define
$$\widetilde{H} :=\left\{\varphi_c \in \text{Aut}(k[X])\, |\, c=(c_1,\dots,c_n)\in (k^*)^n,\ \varphi_c(x_i)=c_i x_i \ \text{and}\ \prod_{t=1}^n c_t^{a_{it}}=1 \  \forall \ i \ \right\},$$
and
$$\widetilde{K} :=\left\{\psi_\sigma \in \text{Aut}(k[X])\ |\ \sigma\in S_n,\ \psi_\sigma(x_i)=x_{\sigma(i)} \ \forall\ i,\ \text{and}\ P_\sigma A = A P_\sigma\right\}.$$ Then $\widetilde{K}$ is a subgroup of $\text{Aut}(k[X])_d$, $\widetilde{H}$ is an abelian normal subgroup of $\text{Aut}(k[X])_d$, and $\text{Aut}(k[X])_d$ is the internal semidirect product $\widetilde{H} \rtimes \widetilde{K}$.
\end{thm}
\begin{proof}
Let $L:=$Aut$(k[X])_d$. Let $e\in \text{Aut}(k[X])_d$ denote the identity automorphism of $k[X]$. We first show that $\widetilde{H}$ is an abelian subgroup of $L$. Note that $e=\varphi_{1}$ where $1=(1,\ldots,1)\in (k^\ast)^n$. Clearly, $\varphi_{1} \in \widetilde{H}.$ Thus $\widetilde{H}\neq \emptyset$. Let $\varphi_c,\varphi_d\in \widetilde{H}$ where $c=(c_1,\dots,c_n), d=(d_1,\dots,d_n)\in(k^*)^n$. Then $(\varphi_c\circ\varphi_d)(x_i)=\varphi_c(d_ix_i)=d_ic_ix_i=c_id_ix_i$. Thus we have $\varphi_c\circ\varphi_d=\varphi_{cd}$, where $cd=(c_1d_1,\dots,c_nd_n)$. Since $\varphi_c,\varphi_d \in \widetilde{H}$, we have that for each $i$, $$\prod_{t=1}^n (c_td_t)^{a_{it}}=\left(\prod_{t=1}^n c_t^{a_{it}}\right)\left(\prod_{t=1}^n d_t^{a_{it}}\right)=1.$$ Similarly, note that $\varphi_c\circ\varphi_d =\varphi_{cd}\in \widetilde{H}$. Thus, $\varphi_d\circ\varphi_c=\varphi_{cd}=\varphi_c\circ\varphi_d$. Thus, any two elements of $\widetilde{H}$ commute. Let $\varphi_c \in \widetilde{H}$ where $c=(c_1,\ldots,c_n)\in (k^\ast)^n$. Recall that for $c=(c_1,\ldots,c_n)\in (k^\ast)^n$, we denote by $c^{-1}=(c_1^{-1},\ldots,c_n^{-1})$ the inverse of $c$ in $(k^\ast)^n$. We have that for each $i$, $$\prod_{t=1}^n (c_t^{-1})^{a_{it}}=\left(\prod_{t=1}^n c_t^{a_{it}}\right)^{-1}=1.$$ Thus $\varphi_{c^{-1}}\in \widetilde{H}$. We have that $\varphi_c\circ\varphi_{c^{-1}}=\varphi_1=\varphi_{c^{-1}}\circ\varphi_c$. Thus, $\varphi_{c^{-1}}$ is the inverse of $\varphi_c$ in $\widetilde{H}$. Therefore, $\widetilde{H}$ is an abelian subgroup of $\text{Aut}(k[X])$.

Next, we show that $\widetilde{K}$ is a subgroup of Aut$(k[X]).$ Note that $e=\psi_{Id}$ where $Id\in S_n$ denote the identity permutation. Clearly, $\psi_{1} \in \widetilde{K}.$ Thus $\widetilde{K}\neq \emptyset$.Let $\psi_{\sigma},\psi_{\tau}\in \widetilde{K}$. Then $P_{\sigma}A=AP_{\sigma}$ and $P_{\tau}A=AP_{\tau}$. We have $(\psi_\sigma\circ\psi_\tau)(x_i)=\psi_\sigma(x_{\tau(i)})=x_{\sigma(\tau(i))}=\psi_{\sigma\tau}(x_i),$ and hence $\psi_\sigma\circ\psi_\tau=\psi_{\sigma\tau}.$ Recall that $P_{\sigma\tau}= P_\sigma P_\tau$. Thus we have $P_{\sigma\tau}A=P_{\sigma}P_{\tau}A=P_{\sigma}AP_{\tau}=AP_{\sigma}P_{\tau}=AP_{\sigma\tau}$, which implies that $\psi_\sigma\circ\psi_\tau=\psi_{\sigma\tau}\in \widetilde{K}$. We have that $\psi_{\sigma}\circ\psi_{\sigma^{-1}}=\psi_{\text{Id}}=\psi_{\sigma^{-1}}\circ\psi_{\sigma}$. Let $\psi_{\sigma}\in \widetilde{K}$. Then we have that $P_{\sigma}A=AP_{\sigma},$ which implies that $A(P_{\sigma})^{-1}= (P_{\sigma})^{-1}A$. Recall that $(P_{\sigma})^{-1}=P_{\sigma^{-1}}$. Thus it follows that $AP_{\sigma^{-1}}=P_{\sigma^{-1}}A$. Thus $\psi_{\sigma^{-1}}\in \widetilde{K}$ is the inverse of $\psi_{\sigma}$. Hence $\widetilde{K}$ is a subgroup of $\text{Aut}(k[X])$.

By Theorem~\ref{iso-thm}, both $\varphi_c$ and $\psi_\sigma$ commute with $d$. Hence $\widetilde{H}$ and $\widetilde{K}$ are subgroups of $L$.

Now let $\rho\in L$. By Theorem~\ref{iso-thm}, there exist $c=(c_1,\dots,c_n)\in (k^*)^n$ and $\sigma\in S_n$ such that $\rho(x_i)=c_i x_{\sigma(i)}$, with  $\prod_{t=1}^n c_t^{a_{it}}=1 \,\ \ \text{and} \,\ \,P_\sigma A=AP_\sigma$. Write $\rho=\varphi_{\sigma^{-1}(c)}\circ \psi_\sigma$, where $\sigma^{-1}(c):=(c_{\sigma^{-1}(1)},\ldots ,c_{\sigma^{-1}(n)})\in (k^*)^n$. Recall that $P_\sigma A=AP_\sigma$ implies that $a_{i\sigma(s)}=a_{\sigma^{-1}(i)s}$ for all $1 \leq i,\, s\leq n$. Since $\varphi_c\in \widetilde{H}$, we have $$\prod_{t=1}^n (c_{\sigma^{-1}(t)})^{a_{it}}=\prod_{s=1}^nc_s^{a_{i\sigma(s)}}=\prod_{s=1}^n c_s^{a_{\sigma^{-1}(i)s}}=1,$$ for all $1\leq i\leq n.$ Hence $\varphi_{\sigma^{-1}(c)} \in \widetilde{H}$ and $\psi_{\sigma}\in \widetilde{K}$. Therefore $L=\widetilde{H}\widetilde{K}$. Let us consider $\rho=\varphi_{\sigma^{-1}(c)}\circ \psi_\sigma=\varphi_{\tau^{-1}(d)} \circ \psi_\tau$, for some $\psi_\sigma, \psi_\tau \in \widetilde{K}$ and $\varphi_{\sigma^{-1}(c)}, \varphi_{\tau^{-1}(d)} \in \widetilde{H}$. Then $(\varphi_{\tau^{-1}(d)})^{-1}\circ \varphi_{\sigma^{-1}(c)}(x_i)=\psi_\tau \circ (\psi_\sigma)^{-1}(x_i)$ for all $i$, that is, $\frac{c_{\sigma^{-1}(i)}}{d_{\tau^{-1}(i)}} x_i= x_{\tau\sigma^{-1}(i)}$. Hence, $\sigma(i)=\tau(i),$ and $c_i=d_i$ for all $i$, that is, $\sigma=\tau$ and $c=d$. Therefore, $\psi_\sigma=\psi_{\tau}$ and $\varphi_{\sigma^{-1}(c)}=\varphi_{\tau^{-1}(d)}$.

Thus it follows that $\widetilde{H}\cap \widetilde{K}=\{e\}$.

Next, we show that $\widetilde{H}$ is a normal subgroup of $L$. Note that by previous paragraph, every element $\rho$ in $L$ can be uniquely written as $\varphi_{\sigma^{-1}(d)}\circ \psi_\sigma$ for some $\varphi_{\sigma^{-1}(d)}\in \widetilde{H}$ and $\psi_\sigma\in \widetilde{K}$ for some $d\in (k^*)^n$ and $\sigma \in S_n$. Thus for $\rho\in L$ and $\varphi_c\in \widetilde{H}$, we have that $\rho\circ\varphi_c\circ\rho^{-1}=(\varphi_{\sigma^{-1}(d)}\circ\psi_\sigma)\circ\varphi_c\circ(\varphi_{\sigma^{-1}(d)}\circ \psi_\sigma)^{-1}$. Thus, to show that $\widetilde{H}\trianglelefteq L$, it is suffices to show that for any $\psi_\sigma\in \widetilde{K}$ and $\varphi_c\in \widetilde{H}$, $\psi_\sigma\circ\varphi_c\circ\psi_\sigma^{-1}\in \widetilde{H}.$ We have $(\psi_\sigma\varphi_c\psi_\sigma^{-1})(x_i)=\psi_\sigma\big(c_{\sigma^{-1}(i)}x_{\sigma^{-1}(i)}\big)=c_{\sigma^{-1}(i)}x_i.$ Hence $\psi_\sigma\varphi_c\psi_\sigma^{-1}=\varphi_{\sigma^{-1}(c)}\in \widetilde{H},$ where $\sigma^{-1}(c)=(c_{\sigma^{-1}(1)},\ldots,c_{\sigma^{-1}(n)}).$ Hence $\widetilde{H}\trianglelefteq L$.

Thus $L=\widetilde{H}\widetilde{K}$, $\widetilde{H}\trianglelefteq L$, and $\widetilde{H}\cap \widetilde{K}=\{e\}$, that is, $L$ is the internal semidirect product of $\widetilde{H}$ and $\widetilde{K}$.
\end{proof}

\begin{rem}\label{rem1}
    Suppose $d$ and $A$ be as in Theorem~\ref{iso-thm} and $\widetilde{H}$ with $\alpha_i=1$ for all $i=1,\dots,n$, and $\widetilde{K}$ be as in Theorem~\ref{sem-dir-product}. Let $$ H := \left\{c=(c_1,\dots,c_n)\in (k^*)^n \,|\,\prod_{t=1}^n c_t^{a_{it}} = 1 \, \text{for all} \, i \right\},$$ and $K := \left\{ \sigma \in S_n \,|\, P_\sigma A = A P_\sigma \right\}$. Then $H$ is a subgroup of $(k^*)^n$ and is isomorphic to $\widetilde{H}$ via the morphism $f :\widetilde{H} \to H$ given by $f(\varphi_c)=c=(c_1,\dots,c_n)$. Also, $K$ (the centralizer of $A$) is a subgroup of $S_n$. Then $\widetilde{K}$ and is isomorphic $K$ via the morphism $g: \widetilde{K}\to K$ given by $g(\psi_\sigma)=\sigma$. Consider the semidirect product $H \rtimes K$, where the action of $K$ on $H$ is defined by $\sigma \cdot c=(c_{\sigma^{-1}(1)},\dots,c_{\sigma^{-1}(n)})$. Then there is an isomorphism $\widetilde{H} \widetilde{K} \cong H\rtimes K$, induced by the isomorphisms of $\widetilde{H}$ with $H$ and $\widetilde{K}$ with $K,$ together with the corresponding actions on the pairs $(\widetilde{H},\widetilde{K})$ and $(H,K)$. In fact, for any $c\in (k^{\ast})^n$ and $\sigma \in S_n$, the conjugation action in the internal semidirect product gives $(\psi_{\sigma}\cdot \varphi_c)(x_i)=\psi_{\sigma}\varphi_c\psi_{\sigma}^{-1}(x_i)=\psi_{\sigma}\varphi_c\psi_{\sigma^{-1}}(x_i)=\psi_{\sigma}(c_{\sigma^{-1}(i)}x_{\sigma^{-1}(i)})=c_{\sigma^{-1}(i)}x_i$, which corresponds to the action $\sigma \cdot c=(c_{\sigma^{-1}(1)},\dots,c_{\sigma^{-1}(n)})$ in the group $H\rtimes K$. 
\end{rem}

Note that the subgroup $H$ is completely described by the matrix $A=[a_{ij}]$. Thus for any matrix $A\in$ M$_{n\times n}(\mathbb{Z})$, as in Theorem~\ref{iso-thm}, we denote the associated subgroup $H$ by $H_A$. Recall that for any matrix $M\in$ M$_n(\mathbb{Z})$ and $c=(c_1,\ldots ,c_n)$, we have $$c^M=\left(\prod_{i=1}^n c_i^{m_{1i}},\dots,\prod_{i=1}^n c_i^{m_{ni}}\right).$$ Thus, for the matrix $A$ the subgroup $H_A=\{c\in (k^\ast)^n\,|\,c^M=1\}$, where $1=(1,\ldots,1)$.

Let $A$ be a matrix as in Theorem~\ref{iso-thm}, and let $H_A$ be the subgroup of $(k^*)^n$ associated with $A$. Suppose, if $c=(c_1,\dots,c_n)\in H_A$ then $c^A=1=(1,\dots,1)$. Now, using the identity adj$(A)A=(\det(A))I$, we get that $(c^A)^{\text{adj}(A)} = c^{(\det (A)I)}$. Hence $c_i^{\det(A)}=1$ for all $1\leq i \leq n$, so $c_i\in \mu_{\det(A)}$. Thus $H_A \subseteq \mu_{\det(A)}^n,$ which is finite. Therefore, $H_A$ is finite.

\begin{thm}\label{final-iso-thm}
    Let $d$, $A$ and $K$ be as in Theorem~\ref{sem-dir-product} and Remark~\ref{rem1}. Let $D=\text{diag}\{d_1,\dots,d_n\}$ be the Smith normal form of $A$ with $d_i\in \mathbb{N}$ for all $i=1,\ldots ,n$. Then $$\text{Aut}(k[X])_d\cong (\mu_{d_1}\times\cdots\times\mu_{d_n})\rtimes K,$$ where $\mu_{d_i}=\{\zeta\in k^\ast\,|\,\zeta^{d_i}=1\}$ is subgroup of $d_i$-th roots of unity, for all $1\leq i \leq n.$
\end{thm}
\begin{proof}
    The proof follows from Theorem~\ref{sem-dir-product}, Remark~\ref{rem1} and Lemma~\ref{struc lem H_A}. 
\end{proof}
Next, we give an application of the \Cref{iso-thm}, c.f., Corollary~\ref{ex1}. Furthermore, as a consequence of Corollary~\ref{ex1} we obtain a partial result of \cite{RK}. First, we recall the following proposition from \cite{NZ}.

\begin{prp}\cite[Proposition 12.4]{NZ}\label{prop1}
   Let $n\geq 3$ and $k[X]:=k[x_1,\dots,x_n]$.  Let $d$ denote the $k$-derivation of $k[X]$  given by $d(x_n)=x_nx_1$ and $d(x_i)=x_ix_{i+1}$ for all $1\leq i\leq n-1$. Then $d$ does not admit any other Darboux polynomial than the products of powers of coordinate functions. As a consequence, the subfield $k(X)^d=k.$
\end{prp}

\begin{cor}\label{ex1}
    Let $n\geq 3$ and $k[X]:=k[x_1,\dots,x_n]$. Let $d$ denote the $k$-derivation of $k[X]$  given by $d(x_n)=x_nx_1$ and $d(x_i)=x_ix_{i+1}$ for all $1\leq i\leq n-1$. Then Aut$(k[X])_d\cong \mathbb{Z}/n\mathbb{Z}$.
\end{cor}

\begin{proof}
    Note that $\beta_{ii}=1$ for all $i=1,\ldots,n$, $\beta_{i(i+1)}=1$ for all $i=1,\ldots ,n-1$, $\beta_{n1}=1$, and $\beta_{ij}=0$ otherwise. Then the matrix $A=[\beta_{ij}-\delta_{ij}]$ becomes
   $$ A = \begin{pmatrix}
    0 & 1 & 0 & \cdots & 0\\    
    0 & 0 & 1 & \cdots & 0\\
    \vdots & \ddots & \ddots & \ddots & \vdots\\
    0 & \cdots & 0 & 0 & 1\\
    1 & 0 & \cdots & 0 & 0
    \end{pmatrix}.$$
    Now, using Proposition~\ref{prop1}, the derivation $d$ satisfies the hypothesis of Theorem~\ref{iso-thm}. Let $D=\text{diag}\{d_1,\dots,d_n\}$ be the Smith normal form of $A$ where $d_i\in \mathbb{N}$. From \cref{final-iso-thm}, we get that $\text{Aut}(k[X])_d\cong  (\mu_{d_1}\times\cdots\times\mu_{d_n}) \rtimes K$, where $K\,\cong\,\{\sigma\in S_n \,|\, P_\sigma A = A P_\sigma\}$. 

    Next, we determine the group $K\,\cong\,\{\sigma\in S_n \,|\, P_\sigma A = A P_\sigma\}$, where $P_\sigma$ denotes the permutation matrix associated with $\sigma$. Note that $A=P_{\tau^{-1}}$ where $\tau=(1\,2\,\dots\,n)\in S_n$. Then $P_\sigma P_{\tau^{-1}} = P_{\tau^{-1}} P_\sigma$, that is, $P_{\sigma\tau^{-1}}=P_{\tau^{-1}\sigma}$ implies that $\sigma\tau^{-1} = \tau^{-1}\sigma$, that is, $\sigma\tau = \tau\sigma$. Then, using the fact that the centralizer of an $n$-cycle in $S_n$ is the cyclic subgroup generated by that cycle, we have that $C_{S_n}(\tau)=\langle\tau\rangle$. Hence $K \,\cong\, C_{S_n}(\tau)\,=\,\langle (1\,2\,\dots\,n)\rangle \,\cong\,\mathbb{Z}/n\mathbb{Z}$.

    Note that the matrix $A$ is unimodular, i.e., its determinant is either $1$ or $-1$. Recall that if $A$ is an unimodular matrix, then the Smith normal form of $A$ is the identity matrix $I$. Thence, $\mu_{d_i}=\{1\}$ for all $i=1,\ldots ,n$.
    
    Then, it follows from \Cref{final-iso-thm} that Aut$(k[X])_d\cong (\mu_{d_1}\times\cdots\times\mu_{d_n}) \rtimes \mathbb{Z}/n\mathbb{Z}$. Since $\mu_{d_i}=\{1\}$ for all $i=1,\ldots ,n$ and $K\cong \mathbb{Z}/n\mathbb{Z}$, we get that Aut$(k[X])_d\cong \mathbb{Z}/n\mathbb{Z}$.
\end{proof}

\begin{rem}\label{LV-remark}
    From Corollary~\ref{ex1} we obtain a partial result of \cite{RK}. In \cite[Theorem 5.2(1)]{RK}, for $n\geq 5$, the authors showed that the Lotka–Volterra type derivation $\tilde{d}(x_i)=x_ix_{i-1}$, for all $i=2,\ldots ,n$, and $\tilde{d}(x_1)=x_1x_n$, has isotropy group isomorphic to $\mathbb{Z}/n\mathbb{Z}$. Note that the derivation $d$, as in Corollary~\ref{ex1},  is conjugate to $\tilde{d}$ via the automorphism $\varphi(x_i)=x_{n-i}$, for all $i=1,\ldots ,n-1$, and $\varphi(x_n)=x_n$. Since isotropy groups are isomorphic under conjugation of derivations, the isotropy group of $d$ is isomorphic with that of $\tilde{d}$. Thus Aut$(k[X])_{\tilde{d}} \cong \mathbb{Z}/n\mathbb{Z}$.
\end{rem}

\begin{cor}\label{thm:ex-1}
Let $d:k[x,y,z]\longrightarrow k[x,y,z]$ be the $k$–derivation defined by $d(x)=x^{c}$, $d(y)=x^{a}yz^{m}$, and  $d(z)=y^{n}z$, where $a,c,m,n\in\mathbb{N}$ with $c\geq 2$. Then $$\text{Aut}(k[x,y,z])_d\cong \mu_{g_1}\times\mu_{\frac{g_2}{g_1}}\times\mu_{\frac{(c-1)mn}{g_2}},$$ where $g_1=\gcd(c-1,a,m,n), g_2=\gcd\big((c-1)m, (c-1)n,an,mn\big)$ and $g=\gcd(c-1,m).$
\end{cor}

\begin{proof}
    Let $\beta_{ij}$ be as in \Cref{iso-thm}. Let $B:= [\beta_{ij}]$ and $A:=[\beta_{ij}-\delta_{ij}]$ be two elements of M$_n(\mathbb{Z})$. Now, 
   $$B =\begin{pmatrix}
    c & 0 & 0\\
    a & 1 & m\\
    0 & n & 1
    \end{pmatrix} \,\,\ \text{and} \,\,\
    A =\begin{pmatrix}
    c-1 & 0 & 0\\
    a & 0 & m\\
    0 & n & 0
    \end{pmatrix},$$ and $\omega_d=$ det$(A)=-(c-1)mn\neq 0$. Also, by \cite[Proposition 11.3]{NZ}, $k(X)^d=k.$ Thus $d$ satisfies all three hypotheses of \Cref{iso-thm}. Then it follows from \Cref{final-iso-thm} that Aut$(k[X])_d\cong  (\mu_{d_1}\times \mu_{d_2} \times\mu_{d_3}) \rtimes K$, where $K\ \cong\ \{\sigma\in S_3 \ |\  P_\sigma A = A P_\sigma\}$ and $d_j\in \mathbb{N}$ for all $j$ such that diag$\{d_1,d_2,d_3\}$ is the Smith normal form of $A$.

    Next, we find the structure of $K$. Since $a\neq0$, it is easy to check that the only permutation matrix commuting with $A$ is the identity matrix. Thus, $K$ is the trivial group.

    Now, we compute the values of $d_i$'s. From \Cref{compute d_i}, we have $d_1=\gcd(c-1,a,m,n)$, $d_1d_2=\gcd\big((c-1)m,(c-1)n,an,mn\big)$ and $d_1d_2d_3=\det(A)=-(c-1)mn$. Let $g_1:=\gcd(c-1,a,m,n)$ and $g_2:=\gcd\big((c-1)m,(c-1)n,an,mn\big)$. Then $$d_1=g_1, \,\ d_2=\frac{g_2}{g_1}, \,\ \text{and} \,\ d_3=\frac{(c-1)mn}{g_2}.$$
    Thus, $\text{Aut}(k[x,y,z])_d\cong \mu_{g_1}\times\mu_{\frac{g_2}{g_1}}\times\mu_{\frac{(c-1)mn}{g_2}}.$
\end{proof}

\begin{cor}\label{thm:ex-2}
Let $k$ be a field of characteristic zero, and let $d:k[x,y,z]\longrightarrow k[x,y,z]$ be the $k$–derivation defined by $d(x)=x^{a}y^{b}$, $d(y)=y^{a}z^{b}$, and $d(z)=z^{a}x^{b}$, where $a,b\in\mathbb{N}$ satisfying $2a\neq b+2$. Then, the isotropy group of $d$ is $$\text{Aut}(k(x,y,z))_d\cong\left(\mu_{g}^2\times\mu_{\frac{(a-1)^3+b^3}{g^2}}\right)\rtimes\mathbb{Z}/3\mathbb{Z},$$ where $g=\text{gcd}(a-1,b)$.
\end{cor}

\begin{proof}
    Let $\beta_{ij}$ be as in \Cref{iso-thm}. Let $B:= [\beta_{ij}]$ and $A:=[\beta_{ij}-\delta_{ij}]$ be two elements of M$_n(\mathbb{Z})$. Now, 
    $$B =\begin{pmatrix}
        a & b & 0\\
        0 & a & b\\
        b & 0 & a
    \end{pmatrix} \,\ \text{and} \,\
      A =\begin{pmatrix}
       a-1 & b   & 0\\
       0   & a-1 & b\\
       b   & 0   & a-1
    \end{pmatrix},$$ and det$(A)=(a-1)^3+b^3\neq 0$. Since $a,b\in \mathbb{N}$ and $2a\neq b+2$, by \cite[Proposition 9.6]{NZ}, $k(X)^d=k.$ Thus $d$ satisfies all three hypotheses of \Cref{iso-thm}. Then it follows from \Cref{final-iso-thm} that Aut$(k[X])_d\cong  (\mu_{d_1}\times \mu_{d_2} \times\mu_{d_3}) \rtimes K$, where $K\,\cong\,\{\sigma\in S_3 \, | \, P_\sigma A = A P_\sigma\}$ and $d_j\in \mathbb{N}$ for all $j$ such that diag$\{d_1,d_2,d_3\}$ is the Smith normal form of $A$.
    
    Next, we find the structure of $K$. Clearly, $(1)\in K$. Moreover, the matrix $P_{\sigma}$ commutes with $A$ if and only if $\sigma \in \{(1),(1 2 3),\, (1 3 2)\}$. Thus $K\cong\langle (1\,2\,3)\rangle\,\cong\,\mathbb{Z}/3\mathbb{Z}$.

    Now, from \Cref{compute d_i}, we get that $d_1=d_2=g$ and $d_3=\frac{(a-1)^3+b^3}{g^2}$ where $g=\gcd(a-1,b)$. Hence, by \Cref{final-iso-thm} we get that $$\text{Aut}(k(x,y,z))_d\cong\left(\mu_{g}^2\times\mu_{\frac{(a-1)^3+b^3}{g^2}}\right)\rtimes\mathbb{Z}/3\mathbb{Z}.$$
\end{proof}

\section{Isotropy Group of Jouanolou-Type Derivations}\label{sec 4}

Let $n\geq 2$. In this section, we consider certain kinds of monomial $k$-derivations $d$ on $k[X]:=k[x_1,\dots,x_n]$ with the property that $x_i$ does not divide $d(x_i)$ for at least one $i\in\{1,2,\dots, n\}$. One such class of examples are {\it Jouanolou} derivations given by $J(n,r):= x_1^r\partial_{x_2}+x_2^r\partial_{x_3}+\cdots +x_n^r \partial_{x_1}$ for $r\geq 2$. In 1979, Jouanolou proved in \cite{J79} that for every $r\geq2$, the derivation $J(3, r)$ has no Darboux polynomial. More generally, for $n\geq 3$ and $r\geq 2$, \.{Z}o{\l}\k{a}dek in \cite{ZH} provided an analytic proof that $J(n,r)$ is without Darboux polynomial. It is known by \cite[Theorem 4.2]{OJN} that the field of rational constants of $J(n,r)$ is trivial if and only if $J(n,r)$ is without a Darboux polynomial. In this section, we consider Jouanolou-type derivations of $k[X]$ and study their isotropy groups.

Let $r\geq 2$ be an integer and $\sigma \in S_n$ be a non-identity permutation. We define the {\it Jouanolou-type} derivation of $k[x_1,\dots,x_n]$ by $$J(n,r,\sigma):=\sum_{i=1}^n x_i^{\,r}\,\partial_{x_{\sigma(i)}}.$$ We would simply use $J$ instead of $J(n,r,\sigma)$ if $n,r,\sigma$ are clear from the context. Note that when $\sigma$ is the $n$ cycle $(12\cdots n)$, we obtain the Jouanolou derivation $J(n,r)$. For a general permutation $\sigma$, $J(n,r,\sigma)$ may or may not contain a Darboux polynomial. For example, consider $n\geq 2,$ $r\geq 2$ and $\sigma\in S_n$ is such that it contains a $2$-cycle, say $(n_1n_2)$ where $1\leq n_1<n_2\leq n$, in its disjoint cycle decomposition. Then $J(n,r,\sigma)(x_{n_1}^{r+1}-x_{n_2}^{r+1})=0$. Thus, we conclude that $x_{n_1}^{r+1}-x_{n_2}^{r+1}$ is a Darboux polynomial and the field of rational constants is not trivial.

\begin{notn}\label{last-not}
We adhere to the following notation throughout this section. Let $n,r\geq 2$ and $\sigma\in S_n$, then the disjoint cycle decomposition of $\sigma$ is denoted by $\sigma=C_1C_2\cdots C_t$, where each cycle of length $m_j$ is given by $C_j=(i_{j,1}\ i_{j,2}\ \dots\ i_{j,m_j})$ for all $1\leq j \leq t$. In the disjoint cycle decomposition, we also consider cycles of length one. The number of cycles of length $m$ appearing in the disjoint cycle decomposition of  $\sigma$ is denoted by $l_m$. For $C_j=(i_{j,1}\ i_{j,2}\ \dots\ i_{j,m_j})$, we define \emph{support} of $C_j$, denoted by Supp$(C_j)$, as the set $\{i_{j,1}\ ,  i_{j,2}\ , \dots\ ,i_{j,m_j}\}$ (Note that this definition agrees with the usual definition of support of a cycle of length greater than one.)  
 
For each $1\le s\le t$, let $X_{m_s}=\{\,x_i \mid i \text{ lies in a cycle of $\sigma$ of length } m_s\,\},$ and let $k[X_{m_s}]$ denote the polynomial subring generated by the variables $x_i$ belonging to $X_{m_s}$. Note that the restriction of the map $J:=\sum_{i=1}^n x_i^{\,r}\,\partial_{x_{\sigma(i)}}$ to the $k$-subalgebra $k[X_{m_s}]$ is also a derivation of the subalgebra, and is denoted by $J|_{k[X_{m_s}]}$. 
\end{notn}
In the following theorem, we give a necessary and sufficient condition for an automorphism of $k[X]$ to lie in the isotropy group of a Jouanolou-type derivation.

\begin{thm}\label{J main}
Let $n,r,\sigma$ be as in Notation \ref{last-not}. Consider the Jouanolou-type derivation $J:=J(n,r,\sigma)=\sum_{i=1}^n x_i^{\,r}\,\partial_{x_{\sigma(i)}}$. Let $\varphi\in Aut(k[X])_J$. Then, for $1\leq i\leq n$, $\varphi(x_i)=\lambda_i\,x_{\tau(i)}$ for some $\tau\in S_n$ and $\lambda_i\in k^*$ satisfying the following conditions
\begin{enumerate}
\item[\emph{(i)}] $\tau$ commutes with $\sigma$, i.e., $\tau\sigma=\sigma\tau$, and
\item[\emph{(ii)}] $\lambda_{\sigma(i)}=\lambda_i^{\,r}$ for $1\leq i\leq n$.
\end{enumerate}
Furthermore, for each cycle $C_j=(i_{j,1}\dots i_{j,m_j})$, we get $\lambda_{i_{j,s}}=\lambda_{i_{j,1}}^{\,r^{\,s-1}},$ where $s$ is an index taken modulo $m_j$.

Conversely, for any $\tau\in S_n$ with $\tau\sigma=\sigma\tau$, and $\lambda_i\in k^\ast$ satisfying $\lambda_{\sigma(i)}=\lambda_i^{\,r}$, the automorphism $\varphi$ of $k[X]$, given by $\varphi(x_i)=\lambda_ix_{\tau(i)}$ for $1\leq i\leq n$, lies in Aut$(k[X])_J$.
\end{thm}

\begin{proof}
Let $\varphi\in$ Aut$(k[X])$ satisfying $J\varphi=\varphi J$. Then, for all $1\leq i \leq n$, we have that $\varphi(x_i)=f_i$ for some $f_i\in k[X]$. Let $d_i:=\deg(f_i)$ be the total degree of $f_i$. Note that $d_i\geq 1$ for all $1\leq i\leq n.$

We have $J(f_i)=J(\varphi(x_i))=\varphi(J(x_i))=\varphi(x_{\sigma^{-1}(i)}^{r})=f_{\sigma^{-1}(i)}^{r}$, for all $1\leq i \leq n$. Furthermore, from the relation $J(f_i)=f_{\sigma^{-1}(i)}^{\,r}$, we get that $\deg J(f_i)=r\,d_{\sigma^{-1}(i)}$. On the other hand, we have that $J(f_i)=\sum_{w=1}^n x_w^{\,r}\,\frac{\partial f_i}{\partial x_{\sigma(w)}}$. Since each $\partial f_i/\partial x_{\sigma(w)}$ has degree at most $d_i-1$, every summand has degree at most $d_i+r-1$. Thus 
\begin{equation}\label{E1}
    rd_{\sigma^{-1}(i)}\leq d_i+r-1
\end{equation} for all $i$. Now, consider a cycle $C_j=(i_{j,1}\dots i_{j,m_j})$ appearing in the disjoint cycle decomposition of $\sigma$. Summing the inequalities of Equation~\ref{E1}, over the indices of $C_j$, we get that $$r\sum_{s=1}^{m_j}d_{i_{j,s}}\leq\sum_{s=1}^{m_j} d_{i_{j,s}} + m_j(r-1).$$
Since $r>1$ and each $d_{i_{j,s}}\geq 1$, it follows that $d_{i_{j,1}}=\cdots=d_{i_{j,m_j}}=1$, that is, each $f_i$ is linear. Thus, for all $i$, we can write $f_i=\sum_{t=1}^n a_{it}x_t + b_i$, where $b_i\in k$, $a_{it}\in k$ for all $t$ such that $a_{it}\neq 0$ for at least one $t$.

Note that $J(f_i)=f_{\sigma^{-1}(i)}^{\,r}$. Rewrite $f_i$ as $f_i=\sum_{w=1}^{n}a_{i\sigma(w)}x_{\sigma(w)}+b_i$. Then for all $i,$ we have $$J(f_i)=\sum_{w=1}^n x_w^{\,r}\,\frac{\partial f_i}{\partial x_{\sigma(w)}}=\sum_{w=1}^n a_{i\sigma(w)}x_w^{r},\;\; \text{and}$$
$$f_{\sigma^{-1}(i)}^{r}=(\sum_{w=1}^{n}a_{\sigma^{-1}(i)\sigma(w)}x_{\sigma(w)}+b_{\sigma^{-1}(i)})^r.$$ Now, comparing the above expressions of $J(f_i)$ and $f_{\sigma^{-1}(i)}^{r}$, we have that for all $i$, 
\begin{equation}\label{eqn J(f_i)=}
\sum_{w=1}^n a_{i\sigma(w)}x_w^{r}=(\sum_{w=1}^{n}a_{\sigma^{-1}(i)\sigma(w)}x_{\sigma(w)}+b_{\sigma^{-1}(i)})^r.
\end{equation} 
Since there is no constant term in L.H.S. of Equation~\ref{eqn J(f_i)=}, it follows that $b_{\sigma^{-1}(i)}=0$ for all $i$, that is, $b_i=0$ for all $i.$ Thus $f_i=\sum_{w=1}^{n}a_{i\sigma(w)}x_{\sigma(w)}$, and Equation~\ref{eqn J(f_i)=} becomes 
\begin{equation}\label{eqn' J(f_i)=}
    \sum_{w=1}^n a_{i\sigma(w)}x_w^{r}=(\sum_{w=1}^{n}a_{\sigma^{-1}(i)\sigma(w)}x_{\sigma(w)})^r,
\end{equation} for all $i.$ Since L.H.S. of Equation~\ref{eqn' J(f_i)=} contains no monomial of mixed terms, it follows that each $f_i$ is a monomial of degree one. Since $\varphi(x_i)=f_i$ for all $1\leq i \leq n$, there exists a permutation $\tau\in S_n$ such that $f_i=\lambda_i x_{\tau(i)}$ for some $\lambda_i \in k^*$. Then $J(f_i)=\lambda_i x_{\sigma^{-1}(\tau(i))}^{\,r}$ and $f_{\sigma^{-1}(i)}^{r}=\lambda_{\sigma^{-1}(i)}^{\,r} x_{\tau(\sigma^{-1}(i))}^{r}$, for all $i$. Since $J(f_i)=f_{\sigma^{-1}(i)}^{r}$, it follows that $\sigma^{-1}(\tau(i))=\tau(\sigma^{-1}(i))$ and $\lambda_i=\lambda_{\sigma^{-1}(i)}^{\,r}$ for all $i$. Hence, conditions (i) and (ii) are satisfied. 

Furthermore, for any cycle $C_j=(i_{j,1}\dots i_{j,m_j})$ appearing in the disjoint cycle decomposition of $\sigma$, we get $\lambda_{i_{j,s}}=\lambda_{i_{j,s-1}}^{\,r}$ for all $2\le s\le m_j$ and $\lambda_{i_{j,1}}=\lambda_{i_{j,m_j}}^{\,r}.$ We show by induction on $s$ that $\lambda_{i_{j,s}}=\lambda_{i_{j,1}}^{\,r^{\,s-1}}$ for all $1\leq s \leq m_j.$ For $s=1$, it is clear. Now assume that $\lambda_{i_{j,{s-1}}}=\lambda_{i_{j,1}}^{\,r^{\,s-2}}$ for some $2\leq s\leq m_j$. Then $\lambda_{i_{j,s}}=\lambda_{i_{j,s-1}}^{r}=(\lambda_{i_{j,1}}^{r^{s-2}})^r=\lambda_{i_{j,1}}^{r^{s-1}}.$

Conversely, assume that $\tau\in S_n$ with $\tau\sigma=\sigma\tau$, and  $\lambda_1,\ldots ,\lambda_n\in k^*$ such that $\lambda_{\sigma(i)}=\lambda_i^{\,r}$. Let us consider the automorphism $\varphi$ of $k[X]$ given by $\varphi(x_i)=\lambda_i x_{\tau(i)}$ for all $i$. Then $$(J\varphi)(x_i)=\lambda_i x_{\sigma^{-1}(\tau(i))}^{\,r}=\lambda_{\sigma^{-1}(i)}^{\,r}\,x_{\tau(\sigma^{-1}(i))}^{\,r}=(\varphi J)(x_i).$$ Thus $\varphi \in$ Aut$(k[X])_J$. 
\end{proof}

\begin{prp}\label{direct sum Aut}
Let $n,r,\sigma$ be as in Notation \ref{last-not}. Consider the Jouanolou-type derivation $J:=J(n,r,\sigma)=\sum_{i=1}^n x_i^{\,r}\,\partial_{x_{\sigma(i)}}$. Let $\varphi\in$ Aut$(k[X])_J$. Let $m_1,\dots,m_q$ be the distinct cycle lengths appearing in the disjoint cycle decomposition of $\sigma$. Then $\varphi|_{k[X_{m_s}]}\in$ Aut$(k[X_{m_s}])_{J|_{k[X_{m_s}]}}$ for all $s=1,\ldots ,q$. Furthermore, the map $$\Psi:\text{Aut}(k[X])_J \longrightarrow \prod_{s=1}^{q} \text{Aut}(k[X_{m_s}])_{J|_{k[X_{m_s}]}}$$ defined by $\Psi(\varphi)=(\varphi|_{k[X_{m_1}]},\dots,\varphi|_{k[X_{m_q}]})$ is a group isomorphism.
\end{prp}

\begin{proof}
    For any $\varphi \in$ Aut$(k[X])_J$, let us denote $\varphi_{m_s}:=\varphi|_{k[X_{m_s}]}$ for all $s=1,\ldots ,q$. For any $\varphi \in$ Aut$(k[X])_J$, from \Cref{J main}, we have
    \begin{equation}\label{E2}
        \varphi(x_i)=\lambda_i x_{\tau(i)},
    \end{equation}
    where $\tau\in S_n$ and $\lambda_i\in k^\ast$ satisfying $\tau\sigma=\sigma\tau$ and $\lambda_{\sigma(i)}=\lambda_i^{\,r}$, for all $i$. Recall that $J(k[X_{m_s}])\subseteq k[X_{m_s}]$ for all $s=1,\ldots,q$.

    {\bf (i)} We first show that $\varphi_{m_s}$ maps to $k[X_{m_s}]$ for all $s=1,\ldots ,q$. Let $i\in\{1,\ldots, n\}$. Suppose $i$ belongs to the support of a cycle of length $m_s$ appearing in the disjoint cycle decomposition of $\sigma$. Then $\sigma^{m_s}(i)=i$. Now, applying $\tau$ and using $\tau\sigma=\sigma\tau$, we have $\sigma^{m_s}(\tau(i))=\tau(\sigma^{m_s}(i))=\tau(i)$. Thus $\tau(i)$ belongs to the support of a cycle of length $m'_s$ appearing in the disjoint cycle decomposition of $\sigma$ where $m'_s$ divides $m_s$. 
    We have $\sigma^{m_{s'}}(\tau(i))=\tau(i)$. Similarly, using $\tau^{-1}\sigma =\sigma \tau^{-1}$, we get that $m_{s}$ divides $m_{s'}$.  Hence, $m_s=m'_s$.
    Now from Equation~\ref{E2}, we get that $\varphi(k[X_{m_s}])\subseteq k[X_{m_s}]$.

    {\bf (ii)} Since $\varphi,\varphi^{-1}\in$ Aut$(k[X])_J$, from (i), we have that $\varphi(k[X_{m_s}])\subseteq k[X_{m_s}]$ and $\varphi^{-1}(k[X_{m_s}])\subseteq k[X_{m_s}]$. Hence it follows that $\varphi_{m_s}\in$ Aut$(k[X_{m_s}])$. Now, we show that $\varphi_{m_s}$ commutes with $J|_{k[X_{m_s}]}$. Let $h\in k[X_{m_s}]$. Since $\varphi_{m_s}\in$ Aut$(k[X_{m_s}])$ and $J(k[X_{m_s}])\subseteq k[X_{m_s}]$, we have that $J|_{k[X_{m_s}]}(\varphi_{m_s}(h))=J(\varphi(h))=\varphi(J(h))=\varphi_{m_s}(J|_{k[X_{m_s}]}(h))$. Hence $J|_{k[X_{m_s}]}\varphi_{m_s}=\varphi_{m_s}J|_{k[X_{m_s}]}$. Thus, $\varphi_{m_s}\in$ Aut$(k[X_{m_s}])_{J|_{k[X_{m_s}]}}$.
   
    {\bf (iii)} Next, we show that $\Psi$ is a bijective group homomorphism. Let $\varphi,\psi\in$ Aut$(k[X])_J$. Then, from (ii), we have that $\varphi_{m_s}, \psi_{m_s} \in$ Aut$(k[X_{m_s}])_{J|_{k[X_{m_s}]}}$ for all $s=1,\ldots ,q$. Hence $(\varphi\psi)_{m_s}=\varphi_{m_s}\psi_{m_s}$ for all $1\leq s \leq q$. Thus $\Psi(\varphi\psi)=\Psi(\varphi)\Psi(\psi)$. Hence, $\Psi$ is a group homomorphism. 

Note that $\varphi$ is the identity map on $k[X]$ if and only if $\varphi_{m_s}$ is the identity map on $k[X_{m_s}]$ for all $1\leq s\leq q$. Hence, the map $\Psi$ is injective. 
    
    Now, we show that $\Psi$ is surjective. Let $(\varphi_{m_1},\dots,\varphi_{m_q})$ be an element of $\prod_{s=1}^{q}$Aut$(k[X_{m_s}])_{J|_{k[X_{m_s}]}}$. Define $\varphi:k[x_1,\dots,x_n]\to k[x_1,\dots,x_n]$ by setting $\varphi(x_i)=\varphi_{m_s}(x_i)$ for $x_i\in X_{m_s}$. Since the sets $X_{m_s}$ are pairwise disjoint and their union is $\{x_1,\ldots,x_n\}$, the map $\varphi$  defines a unique $k$-algebra automorphism. Note that $\varphi_{m_s}\in$ Aut$(k[X_{m_s}])_{J|_{k[X_{m_s}]}}$ and $J(k[X_{m_s}])\subseteq k[X_{m_s}]$ for all $1\leq s\leq q$. Let $x_i\in X_{m_s}$ for some $1\leq s\leq q$. We have $J\varphi(x_i)=J|_{k[X_{m_s}]}\varphi_{m_s}(x_i)=\varphi_{m_s}J|_{k[X_{m_s}]}(x_i)=\varphi J(x_i)$. Hence $J\varphi=\varphi J$. Thus, $\varphi\in$ Aut$(k[X])_J$. Therefore, we have that $\Psi(\varphi)=(\varphi_{m_1},\dots,\varphi_{m_q}),$ proving the surjectivity of $\Psi$.
\end{proof}

Let $\sigma \in$ S$_n$, and the disjoint cycle decomposition of $\sigma$ is given by $\sigma=C_1C_2\cdots C_t$, where $C_j=(i_{j,1}\ i_{j,2}\ \dots\ i_{j,m_j})$. For each $m\geq1$, recall that $l_m$ denotes the number of cycles of length $m$ in the decomposition of $\sigma$. For a fixed $m$, let us define $$H_m:=\mu_{r^m-1}^{l_m}\,\ \ \ \ \text{and} \,\ \ \ \ G_m:=(\mathbb{Z}/m\mathbb{Z})^{l_m}\rtimes S_{l_m},$$ where the action of $S_{l_m}$ on $(\mathbb{Z}/m\mathbb{Z})^{l_m}$ is given by $\rho\cdot (u_1,\ldots,u_{l_m})=(u_{\rho^{-1}(1)},\ldots,u_{\rho^{-1}(l_m)}),$ where $\rho\in S_{l_m}$ and $u=(u_1,\ldots,u_{l_m})\in (\mathbb{Z}/m\mathbb{Z})^{l_m}.$ We note that $G_m$ is nothing but the \emph{generalized symmetric group} $G(m,1,l_m)$. 

Let $(u,\rho)=((u_1,\ldots ,u_{l_m}),\rho)\in G_m$ and $\zeta=(\zeta_1,\dots,\zeta_{l_m})\in H_m$. Consider the map $$\Phi_m:G_m\to \text{Aut}(H_m)$$ given by $$(\Phi_m((u,\rho))(\zeta))_j=\zeta_{\rho^{-1}(j)}^{\,r^{u_j}} \,\,\, \text{for all }\,\,\, j=1,\ldots ,l_m.$$


\begin{lmm}
    The map $\Phi_m((u,\rho))$ is an automorphism of $H_m$ for every $(u,\rho)\in G_m$.
\end{lmm}

\begin{proof}

\textbf{(i)} The map $\Phi_m((u,\rho)):H_m\longrightarrow H_m$ is well-defined.

(a) First, we show that the definition of $\Phi_m((u,\rho))$ is independent of the choice of representatives of the residue classes $u_j\in \mathbb{Z}/m\mathbb{Z}.$ Since $\zeta_j\in \mu_{r^m-1}$, we have $\zeta_j^{\,r^m-1}=1$ for all $1\le j\le l_m$. Hence $\zeta_j^{\,r^m}=\zeta_j,$ and the map $x\to x^{r^m}$ is the identity map on $\mu_{r^m-1}$. Now suppose that $u_j'\equiv u_j \pmod m$. Then $u_j'=u_j+qm$ for some $q\in\mathbb{Z}$, and therefore $r^{u_j'}=r^{u_j+qm}=r^{u_j}(r^m)^q$. Consequently, $\zeta_{\rho^{-1}(j)}^{\,r^{u_j'}}=\zeta_{\rho^{-1}(j)}^{\,r^{u_j}(r^m)^q}=\zeta_{\rho^{-1}(j)}^{\,r^{u_j}},$ since $\zeta_{\rho^{-1}(j)}^{\,r^m}=\zeta_{\rho^{-1}(j)}$ for all $j$. Thus, the expression $\zeta_{\rho^{-1}(j)}^{\,r^{u_j}}$ depends only on the residue class of $u_j$ modulo $m$.

(b) We show that $\Phi_m((u,\rho))(\zeta)\in H_m$. Since $\zeta_j\in\mu_{r^m-1}$ we have that $\zeta_j^{\,r^m-1}=1,$ for all $1\leq j \leq l_m$. Hence $(\zeta_{\rho^{-1}(j)}^{\,r^{u_j}})^{r^m-1}=(\zeta_{\rho^{-1}(j)}^{\,r^m-1})^{r^{u_j}}=1$. Thus every coordinate of $\Phi_m((u,\rho))(\zeta)$ is again a $(r^m-1)$-th root of unity, and therefore $\Phi_m((u,\rho))(\zeta)\in H_m$.

\textbf{(ii)} Next, we show that $\Phi_m((u,\rho)):H_m\longrightarrow H_m$ is a homomorphism. Let $\zeta,\eta\in H_m$. We have $$(\Phi_m((u,\rho))(\zeta\eta))_j= (\zeta \eta)_{\rho^{-1}(j)}^{r^{u_j}}=\zeta_{\rho^{-1}(j)}^{r^{u_j}}\eta_{\rho^{-1}(j)}^{r^{u_j}}=(\Phi_m((u,\rho))(\zeta))_j (\Phi_m((u,\rho))(\eta))_j.$$ It follows that $\Phi_m((u,\rho))$ is a group homomorphism.

\textbf{(iii)} Next, we show that the map $\Phi_m((u,\rho)):H_m\longrightarrow H_m$ given by $\zeta_j\longmapsto\zeta_{\rho^{-1}(j)}^{\,r^{u_j}}$, is bijective. Since $H_m$ is finite, it is enough to show that the map $\Phi_m((u,\rho))$ is injective. For $\zeta \in H_m$, consider $\Phi_m(u,\rho)(\zeta)=(1,\ldots,1)$, that is, $\zeta_{\rho^{-1}(j)}^{r^{u_j}}=1$ for all $j$. By (i)(a), we can assume that $u_j<m$ for each $j$. Then $\gcd(r^{u_j},r^m-1)=1$. Since the order of $\zeta_{\rho^{-1}(j)}$ divides both $r^{u_j}$ and $r^m-1$, and $\gcd(r^{u_j},r^m-1)=1$, it follows that $\zeta_{\rho^{-1}(j)}=1\in \mu_{r^m=1},$ for all $1\leq j \leq l_m$. Thus, the map is injective.

From (i), (ii) and (iii), $\Phi_m((u,\rho)):H_m\to H_m$ is a well-defined bijective group endomorphism of $H_m$, i.e. $\Phi_m((u,\rho))\in\operatorname{Aut}(H_m)$, for every $(u,\rho)\in G_m$.
\end{proof}

Next, we prove the following result, which will be used in \Cref{thm:isotropy}. The above notations will be used for the rest of this section. 

\begin{lmm}\label{lem:semidirect-action}
 Let $H_m$, $G_m$ and $\Phi_m:G_m\longrightarrow \text{Aut}(H_m)$ be as above. Then $\Phi_m$ is a group homomorphism.
\end{lmm}

\begin{proof}
    Recall that the action in the semidirect product $G_m$ is given by $\rho\cdot(v)=\rho\cdot(v_1,\ldots,v_{l_m})=(v_{\rho^{-1}(1)},\ldots,v_{\rho^{-1}(l_m)}),$ where $\rho\in S_{l_m}$ and $v=(v_1,\ldots,v_{l_m})\in (\mathbb{Z}/m\mathbb{Z})^{l_m}$. Let $(u,\rho),(v,\pi)\in G_m$. Then $(u,\rho)(v,\pi)=(u+\rho\cdot(v),\rho\pi)$. We have that 
    \begin{align}\label{eqn Phi}
        (\Phi_m((u,\rho)(v,\pi))(\zeta))_j&=(\Phi_m(u+\rho\cdot(v),\rho\pi)(\zeta))_j \nonumber\\
        &=(\Phi_m((u_1+v_{\rho^{-1}(1)},\ldots,u_{l_m}+v_{\rho^{-1}(l_m)}),\rho\pi)(\zeta))_j \nonumber\\
       & =\zeta_{(\rho\pi)^{-1}(j)}^{r^{u_j+v_{\rho^{-1}(j)}}} \nonumber\\
        &=\zeta_{\pi^{-1}\rho^{-1}(j)}^{r^{u_j+v_{\rho^{-1}(j)}}}.
    \end{align} 
    Now, consider $(\Phi_m(u,\rho)\Phi_m(v,\pi)(\zeta))_j.$ We have that $(\Phi_m(v,\pi)(\zeta))_j=\zeta_{\pi^{-1}(j)}^{r^{v_j}}$. Consider $\eta=(\eta_1,\ldots,\eta_{l_m})\in H_m$ such that $(\eta)_j=\zeta_{\pi^{-1}(j)}^{r^{v_j}}$ for all $1\leq j \leq l_m$. Then 
    \begin{equation}\label{eqn' Phi}
        (\Phi_m(u,\rho)(\Phi_m(v,\pi)(\zeta)))_j=(\Phi_m(u,\rho)(\eta))_j=\eta_{\rho^{-1}(j)}^{r^{u_j}}=(\eta_{\rho^{-1}(j)})^{r^{u_j}}.
    \end{equation}
     Since $\eta_{\rho^{-1}(j)}=\zeta_{\pi^{-1}\rho^{-1}(j)}^{r^{v_{\rho^{-1}(j)}}}$, Equation~\ref{eqn' Phi} becomes 
     \begin{equation}\label{eqn'' Phi}
         (\Phi_m(u,\rho)\Phi_m(v,\pi)(\zeta))_j=(\zeta_{\pi^{-1}\rho^{-1}(j)}^{r^{v_{\rho^{-1}(j)}}})^{r^{u_j}}=\zeta_{\pi^{-1}\rho^{-1}(j)}^{r^{u_j+v_{\rho^{-1}(j)}}}.
     \end{equation}
      Then from Equation~\ref{eqn Phi} and Equation~\ref{eqn'' Phi}, we have $\Phi_m((u,\rho)(v,\pi))=\Phi_m(u,\rho)\Phi_m(v,\pi).$ 
\end{proof}

\begin{thm}\label{thm:isotropy}
Let $n,r,\sigma, l_m$ be as in Notation \ref{last-not}. Let $H_m$, $G_m$ and $\Phi_m$ be as above. Consider the Jouanolou-type derivation $J:=J(n,r,\sigma)=\sum_{i=1}^nx_i^{\,r}\,\partial_{x_{\sigma(i)}}$. Then $$\text{Aut}(k[x_1,\dots,x_n])_J\cong\prod_m\left(\mu_{r^m-1}^{l_m}\rtimes_{\Phi_m}((\mathbb{Z}/m\mathbb{Z})^{l_m}\rtimes S_{l_m})\right),$$ where $m$ varies over the lengths of cycles in the disjoint cycle decomposition of $\sigma$.
\end{thm}
\begin{proof} By Proposition~\ref{direct sum Aut}, it is enough to determine the group Aut$(k[X_m])_{J|_{k[X_m]}}$ for a fixed cycle length $m$. Thus, we can assume that $\sigma=C_1C_2\ldots C_{l_m}$ is the disjoint cycle decomposition of $\sigma,$ where 
$C_j=(i_{j,1},i_{j,2},\ldots, i_{j,m})$ for all $1\leq j \leq {l_m}.$ Throughout the proof, we shall consider the second index $s$ of $i_{j,s}$ under modulo $m$. Recall that $k[X_m]=k[x_{i_{1,1}},\ldots,x_{i_{1,m}},\ldots, x_{i_{l_m,1}},\ldots,x_{i_{l_m,m}}].$

Let $\varphi\in$ Aut$(k[X_m])_{J|_{k[X_m]}}$. Then, by \Cref{J main}, we have that $\varphi(x_i)=\lambda_ix_{\tau(i)}$ for all $i$, where $\tau\in S_{ml_m}$ and $\lambda_i\in k^\ast$ with $\tau\sigma=\sigma\tau$ and $\lambda_{\sigma(i)}=\lambda_i^{\ r}$. Arguing as in the proof of Proposition~\ref{direct sum Aut}(i), we have that for $1 \leq j \leq l_m$, 
$\tau(i_{j,1})= i_{j',1+u_j}$ for some $j'\in\{1,\dots , l_m\}$ and $0\leq u_j \leq m-1$. Then for all $1\leq s\leq m$ it follows that $\tau(i_{j,s})=\tau(\sigma^{s-1}(i_{j,1}))=\sigma^{s-1}(\tau(i_{j,1}))=\sigma^{s-1}(i_{j',1+u_j})=i_{j',s+u_j}.$ Hence 
$\tau(\text{Supp}(C_j))=\text{Supp}(C_{j'})$. Thus, there exists $\rho\in S_{l_m}$ such that $\tau(\text{Supp}(C_j))=\text{Supp}(C_{\rho^{-1}(j)})$ for all $j=1,\ldots,l_m$. 
Hence $\tau(i_{j,s})=i_{\rho^{-1}(j),s+u_j}.$ Since $\varphi$ commutes with $J,$ it follows from \Cref{J main} that $\lambda_{\sigma(i)}=\lambda_i^r$ for all $i$, and thereby $\lambda_{i_{j,s}}=\lambda_{i_{j,1}}^{r^{s-1}}$ for all $1\leq j\leq l_m$ and $1\leq s\leq m.$ Let $\zeta_j:=\lambda_{i_{j,1}}$ for all $j.$ Then using $\lambda_{i}=\lambda_{\sigma^{-1}(i)}^r$ for all $i$, we have that $\zeta_j=\lambda_{i_{j,1}}=\lambda_{i_{j,m}}^r=\zeta_j^{r^m}$. Therefore $\zeta_j^{r^{m}-1}=1,$ and hence $\zeta_j\in \mu_{r^m-1}.$ Thus $\zeta:=(\zeta_1,\ldots,\zeta_{l_m})\in H_m.$ Hence, every $\varphi\in$Aut$(k[X_m])_J$ is uniquely determined by $\zeta=(\zeta_1,\ldots,\zeta_{l_m})\in H_m$, $u=(u_1,\ldots,u_{l_m})\in (\mathbb{Z}/m\mathbb{Z})^{l_m}$ and $\rho\in S_{l_m}$, and is given by
\begin{equation}\label{eqn aut}
    \varphi(x_{i_{j,s}})=\zeta_j^{r^{s-1}}x_{i_{\rho^{-1}(j),s+u_j}}.
\end{equation}
 Recall that $G_m=(\mathbb{Z}/m\mathbb{Z})^{l_m}\rtimes S_{l_m}$ is the semidirect product where $S_{l_m}$ acts on $(\mathbb{Z}/m\mathbb{Z})^{l_m}$ by $\rho\cdot (u_1,\ldots,u_{l_m})=(u_{\rho^{-1}(1)},\ldots,u_{\rho^{-1}(l_m)})$. Next, from Lemma~\ref{lem:semidirect-action} we have the semidirect product $H_m\rtimes_{\Phi_m}G_m$, where the action of $G_m$ on $H_m$ is given by $(\Phi_m(u,\rho)(\zeta))_j=\zeta_{\rho^{-1}(j)}^{r^{u_j}}$ for all $1\leq j \leq l_m.$ 
Let us denote the binary operation on the opposite group $(H_m\rtimes_{\Phi_m} G_m)^{\text{op}}$ by $\star.$ Define $$\Theta:\text{Aut}(k[X_m])_J\longrightarrow (H_m\rtimes_{\Phi_m} G_m)^{\text{op}}$$ by $\Theta(\varphi)=((\zeta_1,\ldots,\zeta_{l_m}),((u_1,\ldots,u_{l_m}),\rho)),$where $\varphi(x_{i_{j,s}})=\zeta_j^{r^{s-1}}x_{i_{\rho^{-1}(j),s+u_j}}$ for all $j=1,\ldots l_m$ and $1\leq s \leq m.$ 

Next, we show that $\Theta$ is a group homomorphism. Let $\varphi,\phi\in$Aut$(k[X_m])_J$. Suppose $\varphi(x_{i_{j,s}})=\zeta_j^{r^{s-1}}x_{i_{\rho^{-1}(j),s+u_j}}$ and $\phi(x_{i_{j,s}})=\eta_j^{r^{s-1}}x_{i_{\pi^{-1}(j),s+v_j}}$, for all $1\leq j\leq l_m$ and $1\leq s\leq m,$ where $(\zeta_{1},\ldots,\zeta_{l_m}),(\eta_{1},\ldots,\eta_{l_m})\in \mu_{r^m-1}^{l_m}$, $(u_{1},\ldots,u_{l_m}),(v_{1},\ldots,v_{l_m})\in (\mathbb{Z}/m\mathbb{Z})^{l_m}$ and $\rho,\pi\in S_{l_m}.$ Then
\begin{align*}
    \varphi\phi(x_{i_{j,s}})=\varphi(\eta_j^{r^{s-1}}x_{i_{\pi^{-1}(j),s+v_j}})=\eta_j^{r^{s-1}}\varphi(x_{i_{\pi^{-1}(j),s+v_j}})&=\eta_j^{r^{s-1}}\zeta_{\pi^{-1}(j)}^{r^{s+v_j-1}}x_{i_{\rho^{-1}(\pi^{-1}(j)},s+v_j+u_{\pi^{-1}(j)}}\\
    &=(\eta_j\zeta_{\pi^{-1}(j)}^{r^{v_j}})^{r^{s-1}}x_{i_{(\pi\rho)^{-1}(j)},s+v_j+u_{\pi^{-1}(j)}}.
\end{align*} Thus we have 
\begin{equation}\label{eqn theta}
    \Theta(\varphi\phi)=((\eta_1\zeta_{\pi^{-1}(1)}^{r^{u_1}},\ldots,\eta_{l_m}\zeta_{\pi^{-1}(l_m)}^{r^{u_{l_m}}}),((v_1+u_{\pi^{-1}(1)},\ldots,v_{l_m}+u_{\pi^{-1}(l_m)}),\pi\rho)).
\end{equation} Let $\zeta:=(\zeta_{1},\ldots,\zeta_{l_m})$, $\eta:=(\eta_{1},\ldots,\eta_{l_m})$, $u_1:=(u_1,\ldots,u_{l_m})$ and $v:=(v_1,\ldots,v_{l_m})$. Then, rewriting Equation~\ref{eqn theta} as
\begin{equation}\label{eqn' theta}
    \Theta(\varphi\phi)=(\eta\Phi_m(v,\pi)(\zeta),((v+\pi\cdot u),\pi\rho)),
\end{equation} the Equation~\ref{eqn' theta} becomes $$\Theta(\varphi\phi)=(\eta\Phi_m(v,\pi)(\zeta),(v+\pi\cdot u),\pi\rho)=(\eta,(v,\pi))(\zeta,(u,\rho))=\Theta(\phi)\Theta(\varphi)=\Theta(\varphi)\star\Theta(\phi).$$ Therefore, $\Theta$ is a group homomorphism. 

Next, we show that $\Theta$ is bijective. By Equation \ref{eqn aut}, we have the uniqueness of $\zeta_j,u_j,\rho$, where $j=1,\ldots,l_m$, which implies that $\Theta$ is injective. To show surjectivity, let $((\zeta_1,\ldots,\zeta_{l_m}),((u_1,\ldots,u_{l_m}),\rho))\in (H_m\rtimes_{\Phi_m} G_m)^{\text{op}}.$ Define a permutation $\tau\in S_{ml_m}$ by $\tau(i_{j,s}):=\sigma^{u_j}(i_{\rho^{-1}(j),s})=i_{\rho^{-1}(j),s+u_j}$, and $\lambda_{i_{j,s}}:=\zeta_j^{r^{s-1}},$ for all $1\leq j \leq l_m$ and $1\leq s \leq m,$ where the index $s$ is taken under modulo $m.$ We first verify that $\tau$ commutes with $\sigma.$ For $1\leq s\leq m$, we have that $\tau\sigma(i_{j,s})=\tau(i_{j,s+1})= i_{\rho^{-1}(j),s+1+u_j}$ and 
$\sigma \tau(i_{j,s})= \sigma(i_{\rho^{-1}(j),s+u_j})= i_{\rho^{-1}(j),s+1+u_j}$.
 Hence $\sigma \tau=\tau\sigma.$ Next, we show that $\lambda_{\sigma(i)}=\lambda_i^r$ for all $i$. For $1\leq s<m,$ $\lambda_{\sigma(i_{j,s})}=\lambda_{i_{j,s+1}}=\zeta_j^{r^{s}}=(\zeta_j^{r^{s-1}})^r=\lambda_{i_{j,s}}^r.$ For $s=m$, since $\zeta_j^{r^m}=\zeta_j,$ we have that $\lambda_{\sigma(i_{j,m})}=\lambda_{i_{j,1}}=\zeta_j=\zeta_j^{r^m}=(\zeta_j^{r^{m-1}})^r=\lambda_{i_{j,m}}^r.$ Thus $\lambda_{\sigma(i)}=\lambda_i^r$ for all $i.$ Thus, $\tau$ and $\lambda_i$ satisfy the hypotheses of the converse part of \Cref{J main}. Hence the map $\varphi$ defined by  $\varphi(x_i)=\lambda_ix_{\tau(i)}$ for all $1\leq i \leq ml_m$, belongs to Aut$(k[X_m])_J$ and $\Theta(\varphi)=((\zeta_1,\ldots,\zeta_{l_m}),((u_1,\ldots,u_{l_m}),\rho))$. Hence, $\Theta$ is surjective. Therefore, Aut$(k[X_m])_J\cong (H_m\rtimes_{\Phi_m} G_m)^{\text{op}}= (\mu_{r^m-1}^{l_m}\rtimes_{\Phi_m}((\mathbb{Z}/m\mathbb{Z})^{l_m}\rtimes S_{l_m}))^{\text{op}}.$ Since every group is isomorphic to its opposite group, we get that Aut$(k[X_m])_J\cong (\mu_{r^m-1}^{l_m}\rtimes_{\Phi_m}((\mathbb{Z}/m\mathbb{Z})^{l_m}\rtimes S_{l_m})).$

\end{proof}

\begin{cor}[\textbf{Jouanolou derivations}]
    Let $\sigma=(1\, 2\,\ldots \,n)\in S_n.$ The corresponding derivation to $\sigma$ is given by $J=x_1^r\partial_{x_2}+x_2^r\partial_{x_3}+\cdots+x_n^r\partial_{x_1}$. Then Aut$(k[x_1,\ldots,x_n])_J\cong\mu_{r^n-1}\rtimes_\Phi \mathbb{Z}/n\mathbb{Z},$ where $\Phi(l)(\zeta)=\zeta^{r^l}$ for $l\in \mathbb{Z}/n\mathbb{Z} $ and $\zeta\in \mu_{r^n-1}$.
\end{cor}

\section*{Acknowledgments}
The first-named author would like to thank the Department of Science and Technology of India for the INSPIRE Faculty fellowship grant IFA23-MA 197 and Anusandhan National Research Foundation (ANRF) for the ARG-MATRICS Grant (ANRF/ARGM/2025/001651/MTR). The second author would like to thank IIT Indore for the Young Faculty Research Grant (No. IITI/YFRSG/2024-25/Phase-VII/05) and Anusandhan National Research Foundation (ANRF) for the PMECRG Grant (ANRF/ECRG/2025/003435/PMS). The third-named author is financially supported by the SRF from CSIR India, Sr. No. 09/1022(16075)/2022-EMR-I.

\bibliographystyle{alpha}

\end{document}